\documentclass[11pt]{article}

\usepackage{fullpage,times,url,natbib}

\usepackage{amsthm,amsfonts,amsmath,amssymb,epsfig,color,float,graphicx,verbatim}
\usepackage{enumitem}
\usepackage{algorithm,algorithmic}
\usepackage{wrapfig}

\usepackage{hyperref}
\hypersetup{
	colorlinks   = true, 
	urlcolor     = blue, 
	linkcolor    = blue, 
	citecolor   = blue 
}

\newtheorem{theorem}{Theorem}
\newtheorem{proposition}{Proposition}
\newtheorem*{proposition*}{Proposition}
\newtheorem{lemma}{Lemma}
\newtheorem{corollary}{Corollary}

\newtheorem{definition}{Definition}
\newtheorem{remark}{Remark}

\newtheorem*{prop:smoothed GD dimension}{\propref{prop:smoothed GD dimension}}

\newcommand{\reals}{\mathbb{R}}
\newcommand{\E}{\mathbb{E}}

\newcommand{\sign}{\mathrm{sign}}

\newcommand{\NN}{\mathbb{N}}
\newcommand{\OO}{\mathbb{O}}
\newcommand{\SSS}{\mathbb{S}}

\newcommand{\be}{\mathbf{e}}
\newcommand{\bx}{\mathbf{x}}
\newcommand{\bw}{\mathbf{w}}
\newcommand{\bg}{\mathbf{g}}

\newcommand{\bu}{\mathbf{u}}
\newcommand{\bv}{\mathbf{v}}
\newcommand{\bz}{\mathbf{z}}
\newcommand{\br}{\mathbf{r}}

\newcommand{\by}{\mathbf{y}}

\newcommand{\bq}{\mathbf{q}}

\newcommand{\Ocal}{\mathcal{O}}
\newcommand{\Acal}{\mathcal{A}}

\newcommand{\Xcal}{\mathcal{X}}

\newcommand{\Fcal}{\mathcal{F}}
\newcommand{\Hcal}{\mathcal{H}}

\newcommand{\Scal}{\mathcal{S}}

\newcommand{\norm}[1]{\|#1\|}
\newcommand{\inner}[1]{\langle#1\rangle}

\newcommand{\secref}[1]{Sec.~\ref{#1}}

\newcommand{\figref}[1]{Fig.~\ref{#1}}
\renewcommand{\eqref}[1]{Eq.~(\ref{#1})}
\newcommand{\lemref}[1]{Lemma~\ref{#1}}

\newcommand{\thmref}[1]{Thm.~\ref{#1}}
\newcommand{\propref}[1]{Proposition~\ref{#1}}
\newcommand{\appref}[1]{Appendix~\ref{#1}}
\newcommand{\remarkref}[1]{Remark~\ref{#1}}

\title{Oracle Complexity in Nonsmooth Nonconvex Optimization}
\author{
 Guy Kornowski  \qquad Ohad Shamir\\
  Weizmann Institute of Science \\
  \texttt{\{guy.kornowski,ohad.shamir\}@weizmann.ac.il}  
}
\date{}

\begin{document}

\maketitle

\begin{abstract}
	It is well-known that given a smooth, bounded-from-below, and possibly nonconvex function, standard gradient-based methods can find $\epsilon$-stationary points (with gradient norm less than $\epsilon$) in $\mathcal{O}(1/\epsilon^2)$ iterations. However, many important nonconvex optimization problems, such as those associated with training modern neural networks, are inherently not smooth, making these results inapplicable. In this paper, we study nonsmooth nonconvex optimization from an oracle complexity viewpoint, where the algorithm is assumed to be given access only to local information about the function at various points. We provide two main results: First, we consider the problem of getting \emph{near} $\epsilon$-stationary points. This is perhaps the most natural relaxation of \emph{finding} $\epsilon$-stationary points, which is impossible in the nonsmooth nonconvex case. We prove that this relaxed goal cannot be achieved efficiently, for any distance and $\epsilon$ smaller than some constants. Our second result deals with the possibility of tackling nonsmooth nonconvex optimization by reduction to smooth optimization: Namely, applying smooth optimization methods on a smooth approximation of the objective function. For this approach, we prove under a mild assumption an inherent trade-off between oracle complexity and smoothness: On the one hand, smoothing a nonsmooth nonconvex function can be done very efficiently (e.g., by randomized smoothing), but with dimension-dependent factors in the smoothness parameter, which can strongly affect iteration complexity when plugging into standard smooth optimization methods. On the other hand, these dimension factors can be  eliminated with suitable smoothing methods, but only by making the oracle complexity of the smoothing process exponentially large.
\end{abstract}

\section{Introduction}

We consider optimization problems associated with functions $f:\reals^d\to\reals$, where $f(\cdot)$ is Lipschitz continuous and bounded from below, but otherwise satisfies no special structure, such as convexity. Clearly, in high dimensions, it is generally impossible to efficiently find a global minimum of a nonconvex function. However, if we relax our goal to finding (approximate) stationary points, then the nonconvexity is no longer an issue. In particular, it is known that if $f(\cdot)$ is \emph{smooth} -- namely, differentiable and with a Lipschitz continuous gradient -- then for any $\epsilon>0$, simple gradient-based algorithms can find $\bx$ such that $\norm{\nabla f(\bx)}\leq \epsilon$, using only $\Ocal(1/\epsilon^2)$ gradient computations, independent of the dimension (see for example \citealt{nesterov2012make,jin2017escape,carmon2019lower}). 

Unfortunately, many optimization problems of interest are inherently \emph{not} smooth. For example, when training modern neural networks, involving max operations and rectified linear units, the associated optimization problem is virtually always nonconvex as well as nonsmooth. Thus, the positive results above, which crucially rely on smoothness, are inapplicable. Although there are positive results even for nonconvex nonsmooth functions, they tend to be either purely asymptotic in nature (e.g., \citealt{benaim2005stochastic,kiwiel2007convergence,zhang2009smoothing,davis2018stochastic,majewski2018analysis}), depend on the internal structure and representation of the objective function, or require additional assumptions which many problems of interest do not satisfy, such as weak convexity or some separation between nonconvex and nonsmooth components\footnote{A trivial example arising in deep learning, which does not satisfy most such structural assumptions, is the negative of the ReLU function, $x\mapsto -\max\{0,x\}$.} (e.g., \citealt{cartis2011evaluation,chen2012smoothing,duchi2018stochastic,bolte2018first,davis2019stochastic,drusvyatskiy2019efficiency,beck2020convergence}). This leads to the interesting question of developing black-box algorithms with non-asymptotic guarantees for nonsmooth nonconvex functions, without assuming any special structure.

In this paper, we study this question via the well-known framework of oracle complexity \citep{nemirovskiyudin1983}: Given a class of functions $\Fcal$, we associate with each $f\in \Fcal$ an \emph{oracle}, which for any $\bx$ in the domain of $f(\cdot)$, returns local information about $f(\cdot)$ at $\bx$ (such as its value and gradient).\footnote{For non-differentiable functions, we use a standard generalization of gradients following \cite{clarke1990optimization}, see \secref{sec:preliminaries} for details.} We consider iterative algorithms which can be described via an interaction with such an oracle: At every iteration $t$, the algorithm chooses an iterate $\bx_t$, and receives from the oracle local information about $f(\cdot)$ at $\bx_t$, which is then used to choose the next iterate $\bx_{t+1}$. This framework captures essentially all iterative algorithms for black-box optimization. In this framework, we fix some iteration budget $T$, and ask what properties can be guaranteed for the iterates $\bx_1,\ldots,\bx_T$, as a function of $T$ and uniformly over all functions in $\Fcal$ (for example, how close to optimal they are, whether they contain an approximately-stationary point, etc.). Unfortunately, as recently pointed out in \citet{zhang2020complexity}, neither small optimization error (in terms of the function value) nor small gradients can be obtained for nonsmooth nonconvex functions with such local-information algorithms: Indeed, approximately-stationary points can be easily ``hidden'' inside some arbitrarily small neighborhood, which cannot be found in a bounded number of iterations. 

Instead, we consider here two alternative approaches to tackle nonsmooth nonconvex optimization, and provide new oracle complexity results for each. We note that \citet{zhang2020complexity} recently proposed another promising approach, by defining a certain relaxation of approximate stationarity (so-called $(\delta,\epsilon)$-stationarity), and remarkably, prove that points satisfying this relaxed goal can be found via simple iterative algorithms with provable guarantees. However, there exist cases where their definition does not resemble a stationary point in any intuitive case, and thus it remains to be seen whether it is the most appropriate one. We further discuss the pros and cons of their approach in Appendix \ref{app:de-stationarity}.

In our first contribution, we consider relaxing the goal of finding approximately-stationary points, to that of finding \emph{near}-approximately-stationary points: Namely, getting $\delta$-close to a point $\bx$ with a (generalized) gradient of norm at most $\epsilon$. This is arguably the most natural way to relax the goal of finding $\epsilon$-stationary points, while hopefully still getting meaningful algorithmic guarantees. Moreover, approaching stationary points is feasible in an asymptotic sense (see for instance \citealt{drusvyatskiy2019efficiency}). Unfortunately, we formally prove that this relaxation already sets the bar too high: For any possibly randomized algorithm interacting with a local oracle, it is impossible to find near-approximately-stationary point with efficient worst-case guarantees, for small enough constant $\delta,\epsilon$.

In our second contribution, we consider tackling nonsmooth nonconvex optimization by reduction to smooth optimization: Given the target function $f(\cdot)$, we first find a smooth function $\tilde{f}(\cdot)$ (with Lipschitz gradients) which uniformly approximates it up to some arbitrarily small parameter $\epsilon$, and then apply a smooth optimization method on $\tilde{f}(\cdot)$. Such reductions are common in convex optimization (e.g., \citealt{nesterov2005smooth,beck2012smoothing,allen2016optimal}), and intuitively, should usually lead to points with meaningful properties with respect to the original function $f(\cdot)$, at least when $\epsilon$ is small enough. For example, it is known that stationary points of $f(\cdot)$ are the limit of approximately-stationary points of appropriate smoothings of $f(\cdot)$, as $\epsilon \rightarrow 0$ \citep[Theorem 9.67]{rockafellar2009variational}.

This naturally leads to the question of how we can find a smooth approximation of a Lipschitz function $f(\cdot)$. Inspecting existing approaches for smoothing nonconvex functions, we notice an interesting trade-off between computational efficiency and the smoothness of the approximating function: On the one hand, there exist optimization-based methods from the functional analysis literature (in particular, Lasry-Lions regularization \citep{lasry1986remark}) which yield essentially optimal gradient Lipschitz parameters, but are not computationally efficient. On the other hand, there exist simple, computationally tractable methods (such as randomized smoothing \citealp{duchi2012randomized}), which unfortunately lead to much worse gradient Lipschitz parameters, with strong scaling in the input dimension. This in turn leads to larger iteration complexity, when plugging into standard smooth optimization methods. Is this kind of trade-off between computational efficiency and smoothness necessary?

Considering this question from an oracle complexity viewpoint, we prove that this trade-off is indeed necessary under mild assumptions: If we want a smoothing method whose oracle complexity is polynomial in the problem parameters, we must accept that the gradient Lipschitz parameter may be no better than that obtained with randomized smoothing (up to logarithmic factors). Thus, in a sense, randomized smoothing is an optimal nonconvex smoothing method among computationally efficient ones. 

It is important to stress that although we formalize our results for general Lipschitz functions, all of our results readily apply to more restricted classes of Lipschitz functions which are often studied in the optimization literature such as Hadamard semi-differentiable, Whitney-stratifiable and semi-algebraic functions. We further discuss this in \remarkref{remark:more assumptions}.

Overall, we hope that our work motivates additional research into black-box algorithms for nonconvex, nonsmooth optimization problems, with rigorous finite-time guarantees.

Our paper is structured as follows. In \secref{sec:preliminaries}, we formally introduce the notations and terminology that we use. In \secref{sec:near-approximate}, we present our results for getting near approximately stationary points. In \secref{sec:smoothing}, we present our results for smoothing nonsmooth nonconvex functions. In \secref{sec:proofs}, we provide full proofs. We conclude in \secref{sec:discussion} with a discussion of open questions. Our paper also contains a few appendices, which beside technical proofs and lemmas, include further discussion of the notion of $(\delta,\epsilon)$-stationarity from \citet{zhang2020complexity} (Appendix \ref{app:de-stationarity}), and a proof that the dimension-dependency arising from randomized smoothing provably affects the iteration complexity of vanilla gradient descent (Appendix \ref{app: smoothed GD dimension}).

\section{Preliminaries}\label{sec:preliminaries}

\textbf{Notation.} We let bold-face letters (e.g., $\bx$) denote vectors. $\bf{0}$ is the zero vector in $\reals^d$ (where $d$ is clear from context), and $\be_1,\be_2,\ldots$ are the standard basis vectors. Given a vector $\bx$, $x_i$ denotes its $i$-th coordinate, and $\bar{\bx}$ denotes the normalized vector $\bx/\norm{\bx}$ (assuming $\bx\neq \mathbf{0}$). $\langle\cdot,\cdot\rangle,~\|\cdot\|$ denote the standard Euclidean dot product and its induced norm over $\reals^d$, respectively. For any real number $x$, we denote $[x]_{+}:=\max\{x,0\}$. Given two functions $f(\cdot),g(\cdot)$ on the same domain $\Xcal$, we define $\|f-g\|_{\infty}=\sup_{\bx\in\Xcal}|f(\bx)-g(\bx)|$. We denote by $\SSS^{d-1}:=\left\{\bx\in\reals^d~\middle|~\|\bx\|=1\right\}$ the unit sphere.
For any natural $N$, we abbreviate $[N]:=\{1,\dots,N\}$. We occasionally use standard big-O asymptotic notation: For functions $f,g:\Xcal\to[0,\infty)$ we write $f=\Ocal(g)$ if there exists $c>0$ such that $f(x)\leq c\cdot g(x)$; $f=\Omega(g)$ if $g=\Ocal(f)$; $f=\Theta(g)$ if $f=\Ocal(g)$ and $g=\Ocal(f)$. We occasionally hide logarithmic factors by writing $f=\tilde{\Ocal}(g)$ if $f=\Ocal(g\log(g+1))$, or $f=\tilde{\Omega}(g)$ if $g=\tilde{\Ocal}(f)$, and also denote by $\text{poly}(\cdot)$ polynomial factors.

\textbf{Gradients and generalized gradients.} If a function $f:\reals^d\to\reals$ is differentiable at $\bx$, we denote its gradient by $\nabla{f}(\bx)$. For possibly non-differentiable functions, we let $\partial f(\bx)$ denote the set of \emph{generalized} gradients (following \citealt{clarke1990optimization}), arguably the most standard extension of gradients to certain classes of nonsmooth nonconvex functions. For Lipschitz functions (which are almost everywhere differentiable by Rademacher's theorem), one simple way to define it is
\[
\partial f(\bx)~:=~ \text{conv}\{\bu:\bu=\lim_{k\rightarrow \infty} \nabla f(\bx_k), \bx_k\rightarrow \bx\}
\]
(namely, the convex hull of all limit points of $\nabla f(\bx_k)$, over all sequences $\bx_1,\bx_2,\ldots$ of differentiable points of $f(\cdot)$ which converge to $\bx$). With this definition, a \emph{stationary point} with respect to $f(\cdot)$ is a point $\bx$ satisfying $\mathbf{0}\in \partial f(\bx)$. Also, given some $\epsilon \geq 0$, we say that $\bx$ is an \emph{$\epsilon$-stationary point} with respect to $f(\cdot)$, if there is some $\bu\in \partial f(\bx)$ such that $\norm{\bu}\leq \epsilon$. 
Furthermore, when some small $\delta,\epsilon$ will be clear from context, we will say that $\bx$ is a near-approximately-stationary point if $\norm{\bx-\bx'}<\delta$ for some $\epsilon$-stationary point $\bx'$.

\textbf{Local oracles.} We consider oracles that given a function $f(\cdot)$ and a point $\bx$, return some quantity $\OO_{f}(\bx)$ which conveys local information about the function near that point. Formally (following \citealt{braun2017lower}), we call an oracle \emph{local} if for any $\bx$ and any two functions $f,g$ that are equal over some neighborhood of $\bx$, it holds that $\OO_{f}(\bx)=\OO_{g}(\bx)$.
An important example is the first order oracle $\OO_{f}(\bx)=(f(\bx),\partial f(\bx))$, but one can consider more sophisticated oracles such as those which return all high order derivative information, wherever they exist.
We choose to work in this generality since our results hold for any local oracle whatsoever, although we note that in the nonconvex-nonsmooth setting, it remains to be seen whether in the general Lipschitz setting discussed in this paper there is any use for any local information which is not first order.

\section{Hardness of Getting Near Approximately-Stationary Points}\label{sec:near-approximate}

In this section, we present our first main result, which establishes the hardness of getting near approximately-stationary points. 

To avoid trivialities, and following \citet{zhang2020complexity}, we will focus on functions $f(\cdot)$ which are Lipschitz and bounded from below. In particular, we will assume that $f(\mathbf{0})-\inf_{\bx}f(\bx)$ is upper bounded by a constant. We note that this is without loss of generality, as $\mathbf{0}$ can be replaced by any other fixed reference point. We consider possibly randomized algorithms which interact with some local oracle. Such an algorithm first produces $\bx_1$ possibly at random, receives $\OO_{f}(\bx_1)$ for some local oracle $\OO_{f}$, and for every $t>1$ produces $\bx_{t}$ possibly at random based on previously observed responses. Our main result in this section is the following:

\begin{theorem}\label{thm:main}
There exist absolute constants $c,C>0$ such that for any algorithm $\Acal$ interacting with a local oracle, and any $T\in\NN,~d\geq2$, there is a function $f(\cdot)$ on $\reals^d$ such that
\begin{itemize}
    \item $f(\cdot)$ is $C$-Lipschitz, $f(\mathbf{0})-\inf_{\bx}f(\bx)\leq C$,
    $\inf\left\{\norm{\bx}~|~\partial f(\bx)=\{\mathbf{0}\}\right\}\leq C$
    .
    \item With probability at least $1-CT\exp(-cd)$ over the algorithm's randomness, the iterates $\bx_1,\ldots,\bx_T$ produced by the algorithm satisfy
    \[
    \inf_{\bx\in\Sigma}\min_{t\in[T]}\norm{\bx_t-\bx}\geq c
    ~,
    \]
    where $\Sigma$ is the set of $c$-stationary points of $f(\cdot)$.
\end{itemize}
\end{theorem}

The theorem implies that it is impossible to obtain worst-case guarantees for finding near-approximately-stationary points of Lipschitz, bounded-from-below functions unless $T$ has exponential dependence on $d$. Note that if an exponential dimension dependence is allowed, then under the theorem's assumptions, one can trivially produce points close to a stationary point (or to any point, for that matter) using an exhaustive grid search. The theorem implies that no oracle-based algorithm will be significantly more efficient than this trivial strategy.
Further notice that since every deterministic algorithm can be viewed as a randomized algorithm with a trivial distribution over its decisions, the theorem above readily holds for deterministic algorithms as well. That being the case, the lower bound described in the second bullet holds deterministically.

\begin{remark}[More assumptions on $f(\cdot)$] \label{remark:more assumptions}
	The Lipschitz functions $f(\cdot)$ used to prove the theorem are based on a composition of affine functions, orthogonal projections, the Euclidean norm function $\bx\mapsto \norm{\bx}$, and the max function. Thus, the result also holds for more specific families of functions considered in the literature, which satisfy additional regularity properties, as long as they contain any Lipschitz functions composed as above. These include, for example, Hadamard semi-differentiable functions \citep{zhang2020complexity}, Whitney-stratifiable functions \citep{bolte2007clarke,davis2018stochastic} and semi-algebraic functions.
\end{remark}

\begin{figure}
\begin{center}
	\includegraphics[trim=0cm 1.2cm 0cm 1.2cm,clip=true, width=0.8\linewidth]{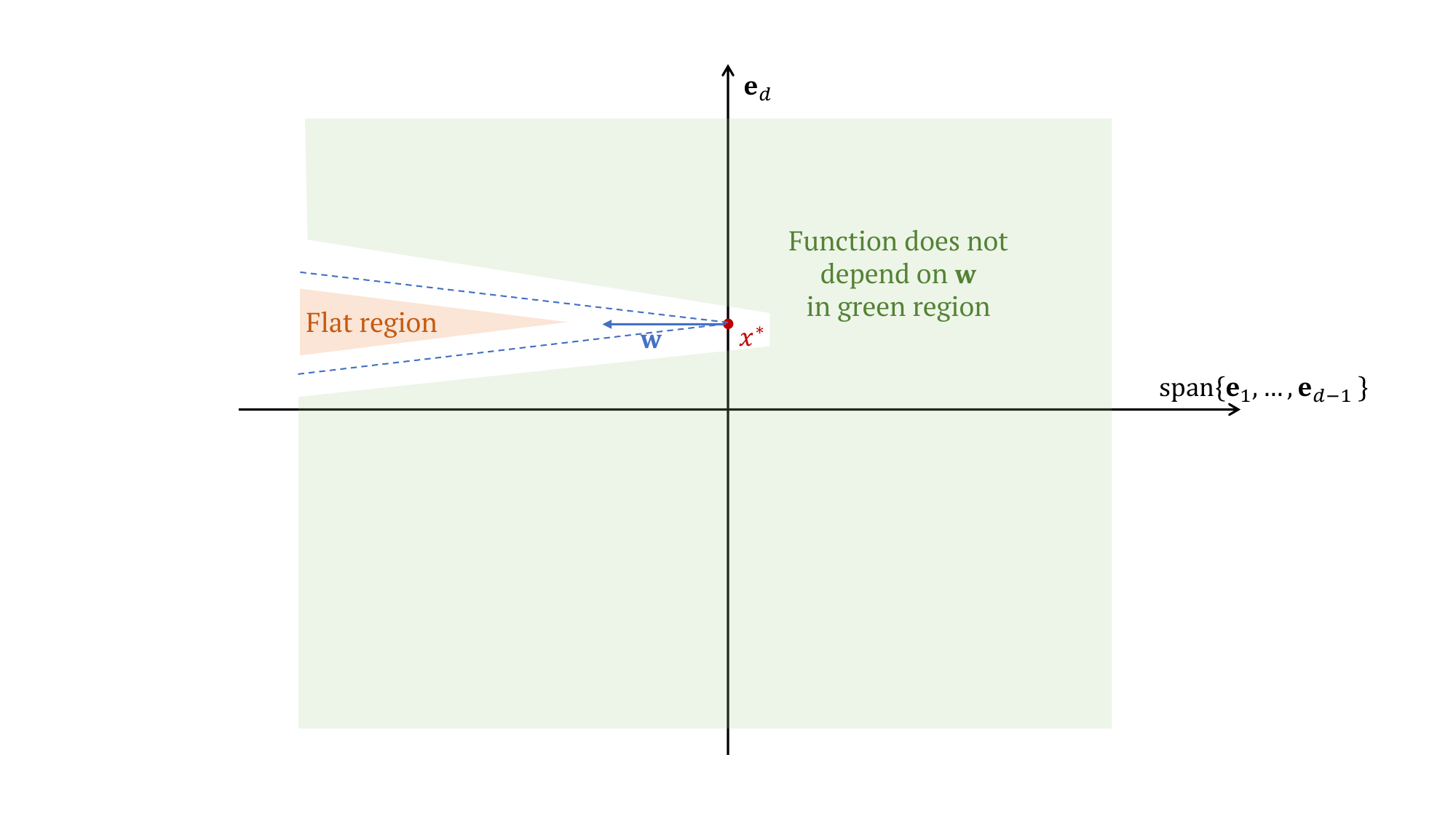}%
	\caption{Illustration of the function
	used in the proof of \thmref{thm:main}.}
	\label{fig:channel}
	\end{center}
\end{figure}

The formal proof of the theorem appears in \secref{sec:proofs}, but can be informally described as follows: First, we construct an algorithm-dependent one dimensional ``hard'' Lipschitz function $h:\reals\to\reals$ that has large derivatives everywhere except for a single point $x^*$. By ``hard'', we mean that after any finite number of steps of the algorithm, we can provide some small neighborhood of $x^*$ which the algorithm is likely not to enter.
Based on $h$, we construct a Lipschitz function on $\reals^d$, specified by a small vector $\bw\in\mathrm{span}\{\be_1,\dots,\be_{d-1}\}$, which resembles the function $\bx\mapsto \norm{(x_1,\dots,x_{d-1})}_2+h(x_d)$ in ``most'' of $\reals^d$, but with a ``channel'' leading away from
a small neighborhood of $x^*\cdot\mathbf{e}_d$ in the direction of $\bw$, and reaching a completely flat region (see \figref{fig:channel}). In high dimensions, the channel and the flat region contain a vanishingly small portion of $\reals^d$. This function has the property of having $\epsilon$-stationary points only in the flat region which is
in the direction of $x^*\cdot\mathbf{e}_d+\bw$, even though the function appears in most places like a function which does not depend on $\bw$. As a result, any oracle-based algorithm that doesn't get to close to $x^*$ in the $d$'th coordinate and doesn't know $\bw$, is unlikely to hit the vanishingly small region where the function differs from $\bx\mapsto \norm{(x_1,\dots,x_{d-1})}_2+h(x_d)$, receiving no information about $\bw$, and thus cannot determine where the $\epsilon$-stationary points lie. As a result, such an algorithm cannot return near-approximately-stationary points.

\section{Smoothing Nonsmooth Nonconvex Functions}\label{sec:smoothing}

In this section, we turn to our second main contribution, examining the possibility of reducing nonsmooth nonconvex optimization to smooth nonconvex optimization, by running a smooth optimization method on a smooth approximation of the objective function. This longstanding approach for nonsmooth nonconvex optimization has been examined by many works, initially driven by practical applications \citep{blake1987visual,wu1996effective} complemented by further theoretical analyses \citep{mobahi2015theoretical,hazan2016graduated}.
In what follows, we focus our discussion on $1$-Lipschitz functions: This is without loss of generality, since if our objective is $L$-Lipschitz, we can simply rescale it by $L$ (and a Lipschitz assumption is always necessary if we wish to obtain a Lipschitz-gradient smooth approximation). Also, we focus on smoothing functions over all of $\reals^d$ for simplicity, but our results and proofs easily extend to the case where we are only interested in smoothing over some bounded domain on which the function is Lipschitz. 

For a nonsmooth \emph{convex} function $f(\cdot)$, a well-known smoothing approach is proximal smoothing (also known as the Moreau envelope or Moreau-Yosida regularization, see \citealp{moreau1965proximite,bauschke2011convex}) defined as $P_{\delta}(f)$ where
\begin{equation}\label{eq:proximal}
P_{\delta}(f)(\bx) := \min_{\by} \left(f(\by)+\frac{1}{2\delta}\norm{\by-\bx}^2\right)~.
\end{equation}
By picking $\delta$ appropriately, $P_{\delta}(f)$ is an arbitrarily good smooth approximation of $f$; more formally, if $f$ is $1$-Lipschitz, then for any $\epsilon>0$, there exists a choice of $\delta=\Theta(\epsilon)$ such that $\|P_{\delta}(f)-f\|_{\infty}\leq \epsilon$, with the gradients of $P_{\delta}(f)$ being $\frac{1}{\epsilon}$-Lipschitz \citep[Section 4.2]{beck2012smoothing}. This is essentially optimal, as no $\epsilon$-approximation can attain a gradient Lipschitz parameter better than $\Omega(1/\epsilon)$ (see \lemref{lemma: L>1/eps} in Appendix \ref{app:technical lemmas} for a formal proof). Finally, computing gradients of $P_{\delta}(f)$ (which can then be fed into a gradient-based smooth optimization method) is feasible, given a solution to \eqref{eq:proximal}, which is a convex optimization problem and hence efficiently solvable in general. 

Unfortunately, for nonconvex functions, proximal smoothing (or other smoothing methods from convex optimization) generally fails in producing smooth approximations. However, it turns out that similar guarantees can be obtained with a slightly more complicated procedure, known as \emph{Lasry-Lions} regularization in the functional analysis literature \citep{lasry1986remark,attouch1993approximation}, which is essentially a double application of proximal smoothing combined with function flipping. One way to define it is as follows: 
\[
P_{\delta,\nu}(f)(\bx) ~:=~ -P_{\delta}(-P_{\nu}(f))(\bx)~=~ \max_{\by} \min_{\bz} \left(f(\bz)+\frac{1}{2\nu}\norm{\bz-\by}^2-\frac{1}{2\delta}\norm{\by-\bx}^2\right)~.
\]
Once more, if $f$ is $1$-Lipschitz, then choosing $\delta,\nu=\Theta(\epsilon)$ appropriately, we get an $\epsilon$-accurate approximation of $f$, with gradients which are $\frac{c}{\epsilon}$-Lipschitz for some absolute constant $c$. However, unlike the convex case, implementing this smoothing involves solving a non-convex optimization problem, which may not be computationally tractable.

Alternatively, a very simple smoothing approach, which works equally well on convex and non-convex problems, is randomized smoothing, or equivalently, convolving the objective function with a smoothness-inducing density function. Formally, given the objective function $f$ and a distribution $P$, we define $\tilde{f}(\bx):=\E_{\by \sim P}[f(\bx+\by)]$. In particular, letting $P$ be a uniform distribution on an origin-centered ball of radius $\epsilon$, the resulting function is an $\epsilon$-approximation of $f(\cdot)$, and its gradient Lipschitz parameter is $\frac{c \sqrt{d}}{\epsilon}$, where $c$ is an absolute constant and $d$ is the input dimension \citep[Lemma 8]{duchi2012randomized}. Moreover, given access to values and gradients of $f(\cdot)$, computing unbiased stochastic estimates of the values or gradients of $\tilde{f}(\cdot)$ is computationally very easy: We just sample a single $\by\sim P$, and return $f(\bx+\by)$ or $\nabla f(\bx+\by)$.\footnote{Recall that by Rademacher's theorem, $f$ is differentiable almost everywhere hence $\nabla f(\bx+\by)$ exists almost surely.} These stochastic estimates can then be plugged into stochastic methods for smooth optimization (see \citealt{duchi2012randomized,ghadimi2013stochastic}).

Comparing these two approaches, we see an interesting potential trade-off between the smoothness obtained and computational efficiency, summarized in the following table:
\begin{center}
	\begin{tabular}{|c||c|c|}
	\hline
	& $\nabla \tilde{f}$ Lipschitz param. & Computationally Efficient?\\\hline\hline
	Randomized Smoothing	&  $c\cdot \sqrt{d}/\epsilon$ & $\checkmark$ \\\hline
	Lasry-Lions Regularization &  $c/\epsilon$ & $\boldsymbol{\times}$\\
	\hline
\end{tabular}
\end{center}
In words, randomized smoothing is computationally efficient (unlike Lasry-Lions regularization), but at the cost of a much larger gradient Lipschitz parameter. Since the iteration complexity of smooth optimization methods strongly depend on this Lipschitz parameter, it follows that in high-dimensional problems, we pay a high price for computational tractability in reducing nonsmooth to smooth problems. As we demonstrate in \appref{app: smoothed GD dimension}, this is a real phenomenon and not just an artifact of iteration complexity analysis, at least for gradient descent.
Roughly speaking, we prove in \propref{prop:smoothed GD dimension}
that when gradient descent is combined with randomized smoothing, it is impossible
to guarantee getting to $\epsilon$-stationary points of the smoothed function within $\Omega(\sqrt{d}/\epsilon)$ iterations.

This discussion leads to a natural question: Is this trade-off necessary, or perhaps there exist computationally efficient methods which can improve on randomized smoothing, in terms of the gradient Lipschitz parameter? Using an oracle complexity framework, we prove that this trade-off is indeed necessary (under mild assumptions), and that randomized smoothing is essentially an optimal method under the constraint of black-box access to the objective $f(\cdot)$, and a reasonable oracle complexity. We note that \citet{duchi2012randomized} proved that the Lipschitz constant cannot be improved by simple randomized smoothing schemes, but here we consider a much larger class of possible methods.

\subsection{Smoothing Algorithms}

Before presenting our main result for this section, we need to carefully formalize what we mean by an efficient smoothing method, since ``returning'' a smooth approximating function over $\reals^d$ is not algorithmically well-defined. Recalling the motivation to our problem, we want a method that given a nonsmooth objective function $f(\cdot)$, allows us to estimate values and gradients of a smooth approximation $\tilde{f}(\cdot)$ at arbitrary points, which can then be fed into standard black-box methods for smooth optimization (hence, we need a uniform approximation property). Moreover, for black-box optimization, it is desirable that this smoothing method operates in an oracle complexity framework, where it only requires local information about $f(\cdot)$ at various points. Finally, we are interested in controlling various parameters of the smoothing procedure, such as the degree of approximation, the smoothness of the resulting function, and the complexity of the procedure. In light of these considerations, a natural way to formalize smoothing methods is the following:
\begin{definition}
	An algorithm $\Scal$ is an $(L,\epsilon,T,M,r)$-smoother if for any $1$-Lipschitz function $f$ on $\reals^d$, there exists a differentiable function $\tilde{f}$ on $\reals^d$ with the following properties:
	\begin{enumerate}[leftmargin=*]
		\item $\|f-\tilde{f}\|_{\infty}\leq \epsilon$, and $\nabla \tilde{f}$ is $L$-Lipschitz.
		\item Given any $\bx\in\reals^d$, the algorithm produces a (possibly randomized) query sequence $\{\bx_1,\dots,\bx_T\}\subset\{\by:\norm{\by-\bx}\leq r\}$, of the form
		$
		\bx_{i+1}=\Scal^{(i)}\left(\xi,\bx,\OO_{f}\left(\bx_{1}\right),\dots,\OO_{f}\left(\bx_{i}\right)\right)
		$,
		where $\Scal^{(i)}$ maps all the previous information gathered by the queries of some local oracle $\OO_{f}$ to a new query, possibly based on a draw of some random variable $\xi$.\footnote{We assume nothing about $\xi$, allowing the algorithm to utilize an arbitrary amount of randomness.} Finally, the algorithm produces a vector
		\[
		\Scal\left(f,{\bf x}\right):=\Scal^{(out)}\left(\xi,\bx,\OO_{f}\left(\bx_{1}\right),\dots,\OO_{f}\left(\bx_{T}\right)\right)~,
		\]
		where $\Scal^{(out)}$ is some mapping to $\reals^{d}$, such that
		\begin{equation} \label{eq:M almost-surely condition}
		\left\|\E_\xi\left[\Scal(f,\bx)\right]-\nabla \tilde{f}(\bx)\right\|\leq \epsilon~~~~\text{and}~~~
		\Pr_\xi\left[\|\Scal(f,\bx)\|\leq M\right]=1~.
		\end{equation}
	\end{enumerate}
\end{definition}

Some comments about this definition are in order. First, the definition is only with respect to the ability of the algorithm to approximate gradients of $\tilde{f}(\cdot)$: It is quite possible that the algorithm also has additional output (such as an approximation of the value of $\tilde{f}(\cdot)$), but this is not required for our results. Second, we do not require the algorithm to return $\nabla \tilde{f}(\bx)$: It is enough that the expectation of the vector output is close to it (up to $\epsilon$). This formulation captures both deterministic optimization-type methods (such as Lasry-Lions regularization, where in general we can only hope to solve the auxiliary optimization problem up to some finite precision) as well as stochastic methods (such as randomized smoothing, which returns $\nabla \tilde{f}(\bx)$ in expectation). Third, we assume that the queries returned by the algorithm lie at a bounded distance $r$ from the input point $\bx$. In the context of randomized smoothing, this corresponds (for example) to using a uniform distribution over a ball of radius $r$ centered on $\bx$. As we discuss later on, some assumption on the magnitude of the queries is necessary for our proof technique. However, requiring almost-sure boundedness is merely for simplicity, and it can be replaced by a high-probability bound with appropriate tail assumptions (e.g., if we are performing randomized smoothing with a Gaussian distribution), at the cost of making the analysis a bit more complicated. 

\begin{remark}\label{remark:regimes}
We are mostly interested (though do not limit our results) to the following parameter regimes:
\begin{itemize}
    \item $T=poly\left(d,L,\epsilon^{-1}\right)$, essentially meaning that a single call to $\Scal$ is computable in a reasonable amount of time.
    \item $M=poly\left(L\right)$. As we formally prove in \lemref{lemma:Lsmooth approx is L-Lip} in Appendix \ref{app:technical lemmas}, if we require $\tilde{f}$ to approximate $f$ and also have $L$-Lipschitz gradients, we must have $\|\nabla\tilde{f}\left(x\right)\|=\Ocal(L)$. In particular, whenever $M$ is sufficiently larger than $L$, \eqref{eq:M almost-surely condition} is interchangeable with the seemingly more natural condition $\Pr_{\xi}\left[\|\Scal(f,\bx)-\nabla\tilde{f}\left(\bx\right)\|\leq M\right]=1$.
    \item $r=\mathcal{O}\left(\epsilon\right)$. If we are interested in smoothing a $1$-Lipschitz,  nonconvex function up to an accuracy $\epsilon$ around a given point $\bx$, we generally expect that only its $\Ocal(\epsilon)$-neighborhood will convey useful information for the smoothing process. We note that this regime is indeed satisfied by randomized smoothing (with a uniform distribution around a radius-$\epsilon$ ball, or with high probability if we use a Gaussian distribution), as well as Lasry-Lions regularization (in the sense that the smooth approximation at $\bx$ does not change if we alter the function arbitrarily outside an $\Ocal(\epsilon)$-neighborhood of $\bx$).  
\end{itemize}
We consider all three of the above to be quite permissive. In particular, notice that randomized smoothing over a ball satisfies the much stronger $T=1,\ M=1,\ r=\epsilon$.\footnote{Indeed, $\Scal(f,\bx):=\nabla f(\bx+\bz)$ where $\bz\sim\mathrm{Unif}(\epsilon\SSS^{d-1})$
requires only a single query (namely $T=1$) at $\bx+\bz$ which is of distance at most $\epsilon$ away from $\bx$ (hence $r=\epsilon$) and satisfies $\|\Scal(f,\bx)\|\leq1$ due to Lipschitzness (hence $M=1$).}
\end{remark}

Our result will require the assumption that the smoothing algorithm $\Scal$ is \emph{translation invariant with respect to constant functions}, in the sense that it treats all constant functions
and regions of the input space in the same manner. We formalize our desired translation-invariance property as follows:
\begin{definition}\label{def:invariant}
	A smoothing algorithm $\Scal$ satisfies $\mathsf{TICF}$ (translation invariance w.r.t. constant functions) if for any two constant functions $f,g$, any $\bx\in\reals^d$ and $i\in[T]$, and any realization of $\xi$,
	\begin{equation} \label{eq:invariant query def}
	\Scal^{(i)}\left(\xi,\bx,\OO_{f}\left(\bx_{1}\right),\dots,\OO_{f}\left(\bx_{i}\right)\right)
	=\Scal^{(i)}\left(\xi,{\bf 0},\OO_{g}\left(\bx_{1}-\bx\right),\dots,\OO_{g}\left(\bx_{i}-\bx\right)\right)+\bx ~,
	\end{equation}
	and
	\begin{equation*}
	    \Scal^{(out)}\left(\xi,\bx,\OO_{f}\left(\bx_{1}\right),\dots,\OO_{f}\left(\bx_{T}\right)\right)=\Scal^{(out)}\left(\xi,{\bf 0},\OO_{g}\left(\bx_{1}-\bx\right),\dots,\OO_{g}\left(\bx_{1}-\bx\right)\right)~.
	\end{equation*}
\end{definition}
In other words, if instead of a constant function $f$ and an input point $\bx$, we pick some other constant function $g$ and the origin, the distribution of the algorithm's sequence of queries remain the same (up to a shift by $-\bx$), and the gradient estimate returned by the algorithm remains the same. 
We consider this to be a mild and natural assumption, which is clearly satisfied by standard smoothing techniques.

\subsection{Main result}

With these definitions in hand, we are finally ready to present our main result for this section, which is the following:
\begin{theorem}\label{thm:smoothingmain}
	Let $\Scal$ be an $(L,\epsilon,T,M,r)$-smoother which satisfies $\mathsf{TICF}$. Then
	\begin{equation} \label{eq:theorem inequality}
	L\sqrt{\log{\left(\left(M+1\right)\left(T+1\right)\right)}}
	\geq
	c_1\cdot\frac{\sqrt{d}}{r}\left(c_2-\epsilon\right)
	\end{equation}
	for some absolute constants $c_1,c_2>0$.
\end{theorem}

This theorem holds for general values of the parameters $L,\epsilon,T,M,r$. Concretely, for parameter regimes of interest (see Remark \ref{remark:regimes}) we have the following corollary:
\begin{corollary}
Suppose that the accuracy parameter satisfies $\epsilon \leq c_2/2$. Then any smoothing algorithm which makes at most $T=poly(d,L,\epsilon^{-1})$ queries at a distance at most $r=\Ocal(\epsilon)$ from the input point, and returns vectors of norm at most $M=poly(d,L,\epsilon^{-1})$, must correspond to a smooth approximation $\tilde{f}(\cdot)$ with Lipschitz gradient parameter \emph{at least} $L=\tilde{\Omega}(\sqrt{d}/\epsilon)$.
\end{corollary}
We note that the lower bound on $L$ in this corollary matches (up to logarithmic factors) the upper bound attained by randomized smoothing. This implies that at least under our framework and assumptions, randomized smoothing is an essentially optimal efficient smoothing method. 

Another implication of \thmref{thm:smoothingmain} is that even if we relax our assumption that $r=\Ocal(\epsilon)$, then as long as $r$ does not scale with the dimension $d$, the gradient Lipschitz parameter of any efficient smoothing algorithm must scale with the dimension (even though there exist dimension-free smooth approximations, as evidenced by Lasry-Lions regularization):
\begin{corollary}
Fix any accuracy parameter $\epsilon\leq c_2/2$, and any $r>0$. Then as long as the number of queries is $T=\text{poly}(d)$ and the output is of size  $M=\text{poly}(d)$, we must have $L\geq \tilde{\Omega}(\sqrt{d})$.
\end{corollary}

A third corollary of our theorem is that (perhaps unsurprisingly), there is no way to implement Lasry-Lions regularization efficiently in an oracle complexity framework: 
\begin{corollary}
If $\epsilon\leq c_2/2$ and $M=\text{poly}(d)$, then any smoothing algorithm for which $\tilde{f}(\cdot)$ corresponds to the Lasry-Lions regularization (which satisfies $L=\Ocal(1)$) must use a number of queries $T=\exp(\tilde{\Omega}(d))$. 
\end{corollary}

\begin{remark}[Dependence on $M$]
Our definition of a smoothing algorithm focuses on the expectation of the algorithm's output. This leads to a logarithmic dependence on $M$ (an upper bound on the algorithm's output) in \thmref{thm:smoothingmain}, since in the proof we need to bound the influence of exponentially-small-probability events on the expectation. It is plausible that the dependence on $M$ can be eliminated altogether, by changing the definition of a smoothing algorithm to focus on the expectation of its output, conditioned on all but highly-improbably events. However, that would complicate the definition and our proof. 
\end{remark}

\begin{figure}[t]
    \centering
    \includegraphics[scale=0.4, trim=150 200 150 240]{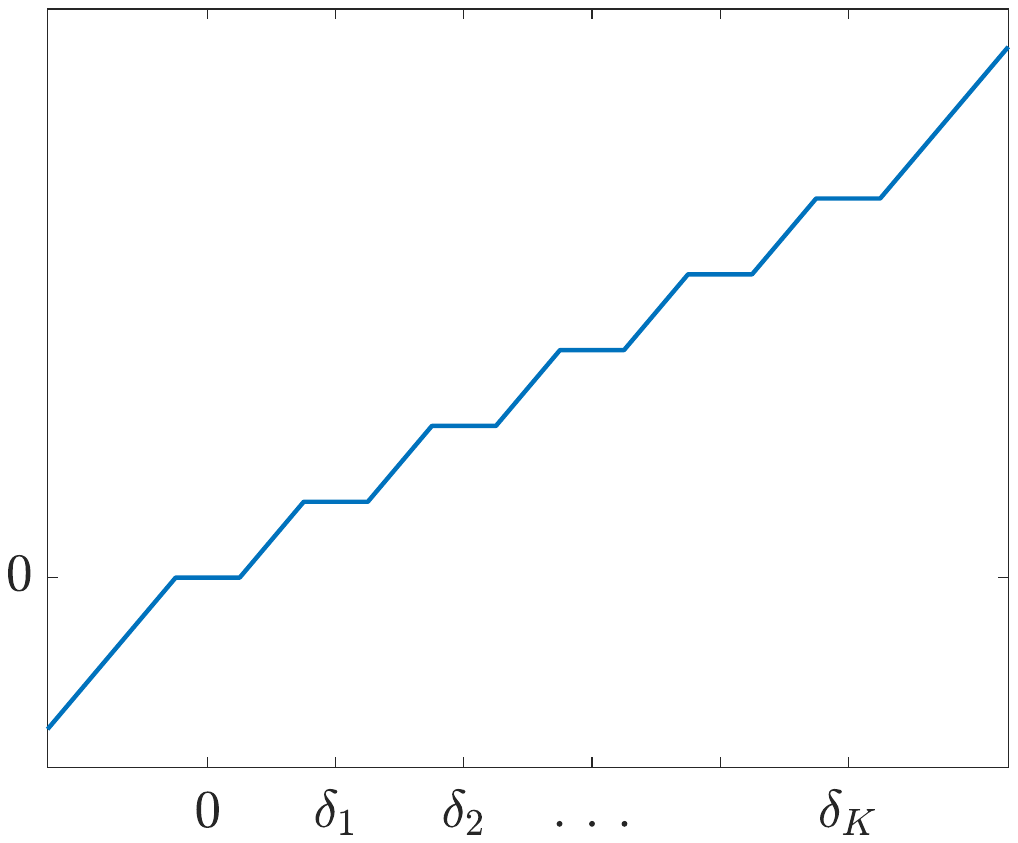}
    \caption{Illustration of $g(x)$, where $\Delta=\left\{0,\delta_1,\dots,\delta_{K}\right\}\subset[0,1]$.}
    \label{fig:g illustration}
\end{figure}
Before presenting the formal proof of \thmref{thm:smoothingmain} in the next section, we outline the main proof idea. Consider a one dimensional monotonically increasing function $g$, which is locally constant at a $\Omega(1/\sqrt{d})$ neighborhood of a grid $\Delta=\{0,\delta_1,\dots,\delta_K\}$ of points in $[0,1]$, with $K$ roughly of order $\sqrt{d}$ (see \figref{fig:g illustration}). We define $f\left(\bx\right)=g(\bw^\top \bx)$, where $\bw\in\SSS^{d-1}$ is a uniformly random unit vector. We note that $f$ is a simple function, easily implemented by (say) a one-hidden layer ReLU neural network, and that it belongs to all of the restricted classes discussed in \remarkref{remark:more assumptions}. 

We now proceed to analyze what happens when a smoothing algorithm is given points of the form $\delta_i\bw$, for $\delta_i\in \Delta$. Since $\bw$ is random, and the algorithm is assumed to be translation invariant, it can be shown (via a concentration of measure argument) that the algorithm is overwhelmingly likely to produce queries in directions which all have $\tilde{\Ocal}(1/\sqrt{d})$ correlation with $\bw$, as long as the number of queries is polynomial. Consequently, with high probability, the queries all lie in a region where the function $f(\cdot)$ is flat, and the algorithm cannot distinguish between it and a constant function. By the translation-invariance property, this implies that the gradient estimates $\nabla \tilde{f}(\delta_i \bw)$ must be of small norm, uniformly for all $\delta_i \bw$. Combining the observation that $\nabla\tilde{f}(\cdot)$ is small along order-of-$\sqrt{d}$-many points between $\bf{0}$ and the unit vector $\bw$, together with the fact that $\nabla\tilde{f}(\cdot)$ is $L$-Lipschitz, we can derive an upper bound on how much $\tilde{f}(\cdot)$ can increase along the line segment between $\mathbf{0}$ and $\bw$, roughly on the order of $L/\sqrt{d}$. On the other hand, $\tilde{f}(\cdot)$ is an approximation of $f$, which has a constant increase between $\bf{0}$ and $\bw$. Overall this allows us to deduce a lower bound on $L$ scaling as $\sqrt{d}$, which results in the theorem.

\section{Proofs}\label{sec:proofs}

\subsection{Proof of \thmref{thm:main}}\label{subsec:proof}

We start by claiming that when optimizing a one dimensional Lipschitz function that has large derivatives everywhere apart from it's global minimum, no algorithm can guarantee this minimum can be \emph{exactly} found with some positive probability within any finite time. The following proposition formalizes this claim.

\begin{proposition} \label{prop: hard 1dim}
For any algorithm $\Acal$ interacting with a local oracle, any $T\in\NN$ and any $\delta>0$, there exists a function $h:\reals\to[0,\infty)$ and $\rho(T,\delta)>0$ such that
\begin{itemize}
    \item $h$ is $2$-Lipschitz, with a single global minimum $x^*\in(0,1)$, and $h(0)=1$.
    \item $\forall x\neq x^*,~\forall g\in\partial h(x):~|g|\geq1$.
    \item The first $T$ iterates produced by $\Acal(h)$: $x_1^{\Acal(h)},\dots,x_T^{\Acal(h)}$, satisfy $\Pr_{\Acal}[\min_{t\in[T]}|x_t-x^*|< \rho]<\delta$.
\end{itemize}
\end{proposition}
\begin{remark}
The quantity $\rho(T,\delta)$ given by \propref{prop: hard 1dim} depends only on $T,\delta$. In particular, it is worth noticing that it is uniform over any algorithm $\Acal$.
\end{remark}
The proof of \propref{prop: hard 1dim} is inspired by a classic lower bound construction for convex optimization due to Nemirovski \citep[Lemma 1.1.1]{nemirovski1995lecture}. The basic idea is to construct two functions that are identical outside some segment, in which their distinct minima lie at distance $2\rho$ one from another. Any algorithm that queries only outside that segment throughout it's first $T-1$ iterates cannot distinguish between the two functions, thus has probability at least $\frac{1}{2}$ to produce it's next iterate at least $\rho$ away from the minimum of the function it is actually optimizing (see \figref{fig:h_1dim_hard} in \appref{app: 1dim hardness proof} for an illustration). By looking at smaller and smaller segments, this idea can be generalized to any number of queries.
Unfortunately, the gradient size of Nemirovski's original construction shrinks exponentially with $T$, therefore cannot be applied to our setting as it does not satisfy the second bullet of \propref{prop: hard 1dim}, which is crucial to the rest of our proof. Nonetheless, we were able to provide a nonconvex construction similar in spirit which satisfies our desired qualities. Due to the substantial length and technicality of the proof, we defer it to \appref{app: 1dim hardness proof}. We continue by showing that given \propref{prop: hard 1dim}, this claim generalizes to any dimension.

\begin{lemma} \label{lem: h provider}
For any algorithm $\Acal$ interacting with a local oracle, any $T\in\NN$, $d\geq2$ and any $\delta>0$, there exists a function $h:\reals\to[0,\infty)$ and $\rho(T,\delta)>0$ such that
\begin{itemize}
    \item $h$ is $2$-Lipschitz, with a single global minimum $x^*\in(0,1)$, and $h(0)=1$.
    \item $\forall x\neq x^*,~\forall g\in\partial h(x):~|g|\geq1$.
    \item If we define $\bar{f}_h(x_1,\dots,x_d):=h(x_d)+\frac{1}{4}\sqrt{\sum_{i=1}^{d-1}x_i^2}$, then the first $T$ iterates produced by $\Acal(\bar{f}_h)$: $\bx_1^{\Acal(\bar{f}_h)},\dots,\bx_T^{\Acal(\bar{f}_h)}$,
    satisfy $\Pr_{\Acal}\left[\min_{t\in[T]}\norm{\bx_t^{\Acal(\bar{f}_h)}-\bx^*}<\rho\right]<\delta$, where $\bx^*:=(0,\dots,0,x^*)$.
\end{itemize}
\end{lemma}
\begin{proof}

Since our goal is to provide a lower bound $\rho$ for optimizing $\bar{f}_h$ with the algorithm $\Acal$, we can assume without loss of generality that $\Acal$ interacts with an even stronger oracle which provides more than just local information about $\bar{f}_h$. Specifically, suppose $\Acal$ has access to an oracle of the form
\[
\overline{\OO}_{\bar{f}_h}(\bx)=\left(
\left\{((z_1,\dots,z_{d-1},x_d),\bar{f}_h(z_1,\dots,z_{d-1},x_d))\middle|(z_1,\dots,z_{d-1})\in\reals^{d-1}\right\},\OO_{h}(x_d)
\right)
\]
for some local oracle $\OO$. Namely, a full global description of the function $\bar{f}_h$ over the affine subspace $\{\bz|~z_d=x_d\}$, coupled with local information with respect to the last coordinate. Note that by definition of $\bar{f}_h(\bx)$, changing $x_d$ only affects $h(x_d)$, thus all the information about $x_d$ is indeed conveyed through $h(x_d)$.

Assuming $\Acal$ interacts with the described $\overline{\OO}$, we turn to describe \emph{another} algorithm $\Acal'$, which given local oracle access to a one dimensional function $h$ works as follows:
\begin{itemize}
    \item $\Acal'$ simulates $\Acal$ and receives it's first iterate $\bx_1$. It then produces the first iterate $(\bx_1)_d$.
    \item Given $\Acal$'s current iterate $\bx_t$, $\Acal'$ queries $\OO_{h}((\bx_t)_d)$. Then, $\Acal'$ feeds into $\Acal$
    \[
    \left(
    \left\{((z_1,\dots,z_{d-1},(\bx_t)_d),\bar{f}_h(z_1,\dots,z_{d-1},(\bx_t)_d)\middle|(z_1,\dots,z_{d-1})\in\reals^{d-1}\right\},\OO_{h}((\bx_t)_d)
    \right)~,
    \]
    and receives from $\Acal$ it's next iterate $\bx_{t+1}$. $\Acal'$ then produces it's next iterate $(\bx_{t+1})_d$.
\end{itemize}

It is clear that $\Acal'$ is indeed a well defined algorithm which interacts with a local oracle. Hence, let $h(\cdot),\rho(T,\delta)$ be the function and the positive parameter given by \propref{prop: hard 1dim} for $\Acal',T,\delta$. We will show these $h,\rho$ satisfy all three bullets in the lemma. The first two bullets are immediate, thus it only remains to prove the third. To that end, note that by construction of $\Acal'$, $\Acal'$'s iterates when applied to $h$ are exactly the $d$'th coordinates of $\Acal$ when applied to $\bar{f}$. That is, $(\bx_t^{\Acal(\bar{f})})_d=x_t^{\Acal'(h)}$. Consequently,
\begin{align*}
\Pr_{\Acal}\left[\min_{t\in[T]}\norm{\bx_t^{\Acal(\bar{f})}-\bx^*}<\rho\right]
&\leq
\Pr_{\Acal}\left[\min_{t\in[T]}|(\bx_t^{\Acal(\bar{f})})_{d}-(\bx^*)_d|<\rho\right] \\
&=\Pr_{\Acal'}\left[\min_{t\in[T]}|x_t^{\Acal'(h)}-x^*|<\rho\right]<\delta~,
\end{align*}
where the last inequality follows by definition of $h,\rho$.
\end{proof}

From now on, we fix some algorithm $\Acal$ interacting with a local oracle, $T\in\NN$, $d\geq2$ and set $\delta=T\exp(-d/36)$. We denote by $h(\cdot),~\rho>0,~\bx^*\in\reals^d$ their associated function, positive parameter and point given by \lemref{lem: h provider}.
Given any nonzero vector $\bw\in\reals^d$ such that $w_{d}=0$, we define 
\[
f_{\bw}(x_1,\dots,x_d):=h(x_d+x^*)+\frac{1}{4}\sqrt{\sum_{i=1}^{d-1}x_i^{2}}-\left[\langle\overline{\bw},\bx+\bw\rangle-\frac{1}{2}\|\bx+\bw\|\right]_{+}~,
\]
where $\overline{\bw}:=\bw/\norm{\bw}$.
This function looks like the ``hard'' function given by \lemref{lem: h provider} (up to a shift along the $d$'th axis) as long as $\bx$ is not highly correlated with $\bw$, which makes the ReLU term vanish. However, unlike the hard function which has a stationary point, the ReLU term adds at that point a large gradient component in the $\bw$ direction, preventing it from being even $\epsilon$-stationary. Moreover, the following lemma shows that this function has no $\epsilon$-stationary points for small $\epsilon$.

\begin{lemma} \label{lem: f_w properties}
For any nonzero $\bw\in\reals^d$ such that $w_d=0$, $f_{\bw}$ is $\frac{15}{4}$-Lipschitz and has no $\epsilon$-stationary points for any $\epsilon<\frac{1}{4\sqrt{2}}$.
\end{lemma}
\begin{proof}
Throughout the proof we omit the $\bw$ subscript and refer to $f_{\bw}(\cdot)$ as $f(\cdot)$.

The functions $\bx\mapsto h(x_d+x^*)$, $\bx\mapsto\frac{1}{4}\sqrt{\sum_{i=1}^{d-1}x_i^{2}}$, $\bx\mapsto\langle\overline{\bw},\bw+\bx\rangle$, $\bx\mapsto\frac{1}{2}\|\bx+\bw\|$ and $x\mapsto[x]_+$ are $2$-Lipschitz, $\frac{1}{4}$-Lipschitz, $1$-Lipschitz, $\frac{1}{2}$-Lipschitz and $1$-Lipschitz respectively, from which it follows that $f$ is $2+\frac{1}{4}+1+\frac{1}{2}=\frac{15}{4}$-Lipschitz.

In order to prove that no point $\bx$ is $\epsilon$-stationary for any $\epsilon<\frac{1}{4\sqrt{2}}$, we will use the facts that $\partial(g_1+g_2)\subseteq \partial g_1 + \partial g_2$, and that if $g_1$ is univariate, $\partial (g_1\circ g_2)(\bx)\subseteq \text{conv}\{r_1 \br_2:r_1\in \partial g_1(g_2(\bx)),\br_2\in \partial g_2(\bx)\}$ (see \citealp[Proposition 2.3.3 and Theorem 2.3.9]{clarke1990optimization}). We examine six exhaustive cases:
\begin{itemize}
    \item $x_{d}\neq0$. In this case we have
    \[
    \partial f(\bx)\subseteq
    \left\{t\cdot\mathbf{e}_{d}+\frac{1}{4}\bu-s\left(\overline{\bw}-\frac{1}{2}\bv\right)
    \middle|
    ~|t|\geq1,~u_d=0,~s\in[0,1],~\|\bv\|\leq1\right\}~.
    \]
    For any $\bg\in\partial f(\bx)$ corresponding to some $t,\bu,s,\bv$, using the fact that $u_d=w_d=0$ we get
    \[
    \|\bg\|\geq\langle\bg,\sign(t)\cdot\mathbf{e}_{d}\rangle\geq|t|-\frac{1}{2}\geq\frac{1}{2}>\frac{1}{4\sqrt{2}}~.
    \]
    
    \item $\bx=\bf{0}$. In this case we have
    \[
    \partial f(\bx)\subseteq
    \left\{
    t\cdot\mathbf{e}_d+\frac{1}{4}\bu-\overline{\bw}+\frac{1}{2}\overline{\bw}
    \middle|
    ~t\in[-2,2],~\sum_{i=1}^{d-1}u_i^2\leq1,~u_{d}=0
    \right\}~.
    \]
    For any vector in $\partial f(\bx)$ corresponding to some $t,\bu$, we use the fact that projecting any vector onto $\mathrm{span}\{\mathbf{e}_1,\dots,\mathbf{e}_{d-1}\}$ cannot increase it's norm, in order to get
    \[
    \left\|t\cdot\mathbf{e}_d+\frac{1}{4}\bu-\frac{1}{2}\overline{\bw}\right\|
    \geq\left\|\frac{1}{4}\bu-\frac{1}{2}\overline{\bw}\right\|
    =\left\|\frac{1}{2}\overline{\bw}-\frac{1}{4}\bu\right\|
    \geq\frac{1}{2}-\frac{1}{4}>\frac{1}{4\sqrt{2}}~.
    \]
    
    \item $\bx=-\bw$. In this case we have
    \[
    \partial f(\bx)\subseteq
    \left\{
    t\cdot\mathbf{e}_d-\frac{1}{4}\overline{\bw}-s\left(\overline{\bw}-\frac{1}{2}\bu\right)
    \middle|
    ~t\in[-2,2],~s\in[0,1],~\|\bu\|\leq1
    \right\}~.
    \]
    For any $\bg\in\partial f(\bx)$ corresponding to some $t,s,\bu$, using the fact that $w_d=0$ we get
    \begin{align*}
    \|\bg\|\geq\langle\bg,-\overline{\bw}\rangle
    &=\left\langle-\frac{1}{4}\overline{\bw},-\overline{\bw}\right\rangle
    +\left\langle-s\overline{\bw},-\overline{\bw}\right\rangle
    +\left\langle\frac{s}{2}\bu,-\overline{\bw}\right\rangle \\
    &=\frac{1}{4}+s-\frac{s}{2}\langle\bu,\bw\rangle
    \geq\frac{1}{4}+s-\frac{s}{2}\geq\frac{1}{4}>\frac{1}{4\sqrt{2}}~.
    \end{align*}
    
    \item $x_{d}=0$, $\bx\notin\{\bf{0},-\bw\}$, $\langle\overline{\bw},\overline{\bx+\bw}\rangle<\frac{1}{2}$. Note that
    \[
    \langle\overline{\bw},\overline{\bx+\bw}\rangle<\frac{1}{2}
    \implies\langle\overline{\bw},\bx+\bw\rangle-\frac{1}{2}\|\bx+\bw\|<0~,
    \]
    and that the set $\left\{\bx\middle|~\langle\overline{\bw},\overline{\bx+\bw}\rangle<\frac{1}{2}\right\}\setminus\left\{\mathbf{0},-\bw\right\}$ is an open set in $\reals^d$. Thus for every such point, the function $\bx\mapsto f(\bx)$ is locally identical to $\bx\mapsto h(x_d+x^*)+\frac{1}{4}\sqrt{\sum_{i=1}^{d-1}x_i^{2}}$, which in particular implies that their gradient sets are identical. Furthermore, combining the assumptions $x_d=0,\bx\neq\mathbf{0}$ reveals that $x_1,\dots,x_{d-1}$ are not all zeroes. Consequently, we get
    \[
    \partial f(\bx)\subseteq
    \left\{
    t\cdot\mathbf{e}_d+\frac{1}{4}\overline{\bx}
    \middle|
    ~t\in[-2,2]
    \right\}~.
    \]
    For any $\bg\in\partial f(\bx)$ corresponding to some $t$, we get
    \[
    \|\bg\|\geq\langle\bg,\overline{\bx}\rangle=\frac{1}{4}>\frac{1}{4\sqrt{2}}~.
    \]
    
    \item $x_{d}=0$, $\bx\notin\{\bf{0},-\bw\}$, $\langle\overline{\bw},\overline{\bx+\bw}\rangle>\frac{1}{2}$. Note that 
    \[
    \langle\overline{\bw},\overline{\bx+\bw}\rangle>\frac{1}{2}
    \implies\langle\overline{\bw},\bx+\bw\rangle-\frac{1}{2}\|\bx+\bw\|>0~,
    \]
    and that the set $\left\{\bx\middle|~\langle\overline{\bw},\overline{\bx+\bw}\rangle>\frac{1}{2}\right\}\setminus\left\{\mathbf{0},-\bw\right\}$ is an open set in $\reals^d$. Thus for every such point, the function $\bx\mapsto f(\bx)$ is locally identical to 
    \[
    \bx\mapsto h(x_d+x^*)+\frac{1}{4}\sqrt{\sum_{i=1}^{d-1}x_i^{2}}-\langle\overline{\bw},\bx+\bw\rangle+\frac{1}{2}\|\bx+\bw\|~,
    \]
    which in particular implies that their gradient sets are identical. Furthermore, combining the assumptions $x_d=0,\bx\neq\mathbf{0}$ reveals that $x_1,\dots,x_{d-1}$ are not all zeroes. Consequently, we get
    \[
    \partial f(\bx)\subseteq
    \left\{
    t\cdot\mathbf{e}_d+\frac{1}{4}\overline{\bx}-\overline{\bw}+\frac{1}{2}\left(\overline{\bx+\bw}\right)
    \middle|
    ~t\in[-2,2]
    \right\}~.
    \]
    For any $\bg\in\partial f(\bx)$ corresponding to some $t$, using the fact that $w_d=0$ we get
    \begin{align*}
    \|\bg\|&\geq\langle\bg,-\overline{\bw}\rangle
    =\frac{1}{4}\left\langle\overline{\bx},-\overline{\bw}\right\rangle
    +\left\langle-\overline{\bw},-\overline{\bw}\right\rangle
    +\frac{1}{2}\left\langle\overline{\bx+\bw},-\overline{\bw}\right\rangle\\
    &\geq-\frac{1}{4}+1-\frac{1}{2}>\frac{1}{4\sqrt{2}}~.
    \end{align*}
    
    \item $x_{d}=0$, $\bx\notin\{\bf{0},-\bw\}$, $\langle\overline{\bw},\overline{\bx+\bw}\rangle=\frac{1}{2}$. In this case we have
    \begin{align}
    &\partial f(\bx)\subseteq
    \left\{
    t\cdot\mathbf{e}_d+\frac{1}{4}\overline{\bx}-s\left(\overline{\bw}-\frac{1}{2}(\overline{\bx+\bw})\right)
    \middle|
    ~t\in[-2,2],~s\in[0,1]
    \right\} \label{eq: g condition}\\
    &=\left\{
    \left(\frac{1}{4\|\bx\|}+\frac{s}{2\|\bx+\bw\|}\right)\bx+\left(\frac{s}{2\|\bx+\bw\|}-\frac{s}{\|\bw\|}\right)\bw+t\cdot\mathbf{e}_d
    \middle|
    ~t\in[-2,2],~s\in[0,1]
    \right\}~. \nonumber
    \end{align}
    Denote $\bx=\bx_{\mid}+\bx_{\perp}$ where $\bx_{\perp}=(I-\bar{\bw}\bar{\bw}^T)\bx$ is the orthogonal projection of $\bx$ onto $\mathrm{span}(\bw)^{\perp}$, and $\bx_{\mid}\in\mathrm{span}(\bw)$.
    For any $\bg\in\partial f(\bx)$ corresponding to some $t,s$, using the fact that $x_d=w_d=0$ we get for some scalar $\alpha$:
    \begin{align*}
        \|\bg\|
        &\geq\left\|\left(\frac{1}{4\|\bx\|}+\frac{s}{2\|\bx+\bw\|}\right)\bx+\left(\frac{s}{2\|\bx+\bw\|}-\frac{s}{\|\bw\|}\right)\bw\right\| \\
        &=\left\|\left(\frac{1}{4\|\bx\|}+\frac{s}{2\|\bx+\bw\|}\right)\bx_{\perp}+\alpha\cdot\bw\right\| \\
        &\geq\left(\frac{1}{4\|\bx\|}+\frac{s}{2\|\bx+\bw\|}\right)\left\|\bx_{\perp}\right\| \\
        &\geq\frac{1}{4\|\bx\|}\cdot\|\bx_{\perp}\|~.
    \end{align*}
    Since $(I-\bar{\bw}\bar{\bw}^T)$ is an orthogonal projection, in particular symmetric, we also have
		\[
		\norm{\bx_{\perp}}^2
		=\langle\bx,(I-\bar{\bw}\bar{\bw}^T)^2 \bx\rangle
		=\langle\bx,(I-\bar{\bw}\bar{\bw}^T) \bx\rangle
		=\norm{\bx}^2-\langle\overline{\bw},\bx\rangle^2
		=\norm{\bx}^2(1-\langle\overline{\bw},\overline{\bx}\rangle^2)~.
		\]
		Plugging into the above, it follows that $\|\bg\|$ is at least $\frac{1}{4}\sqrt{1-\langle\overline{\bw},\overline{\bx}\rangle^2}$. Assuming that there exists such $\bg\in\partial f(\bx)$ with norm at most $\epsilon$, it follows that
		\begin{equation} \label{eq: wx}
		\frac{1}{4}\sqrt{1-\langle\overline{\bw},\overline{\bx}\rangle^2}
		\leq\epsilon~.
		\end{equation}
		However, we will show that for any $\epsilon<\frac{1}{4\sqrt{2}}$, we must arrive at a contradiction. To that end, let us consider two cases:
		\begin{itemize}
			\item If $\langle\overline{\bw},\overline{\bx}\rangle>0$, then by rearranging \eqref{eq: wx}, we have $\langle\overline{\bw},\overline{\bx}\rangle\geq\sqrt{1-16\epsilon^2}$. Hence,
			\[
			\langle\overline{\bw},\bx+\bw\rangle
			\geq\norm{\bx}\sqrt{1-16\epsilon^2}+\norm{\bw}
			\geq(\norm{\bx}+\norm{\bw})\sqrt{1-16\epsilon^2}
			\geq\norm{\bx+\bw}\sqrt{1-16\epsilon^2}~.
			\]
			However, dividing both sides by $\norm{\bx+\bw}$, we get that $\langle\overline{\bw},\overline{\bx+\bw}\rangle\geq\sqrt{1-16\epsilon^2}$. If $\epsilon<\frac{1}{4\sqrt{2}}$, we get that $\langle\overline{\bw},\overline{\bx+\bw}\rangle>\frac{1}{2}$, contradicting our assumption on $\bx$.
			
			\item If $\langle\overline{\bw},\overline{\bx}\rangle\leq0$, then by \eqref{eq: wx}, we must have $\langle\overline{\bw},\overline{\bx}\rangle\leq-\sqrt{1-16\epsilon^2}$. But then, by recalling that $w_d=0$, we use \eqref{eq: g condition} and our assumption that $\langle\overline{\bw},\overline{\bx+\bw}\rangle=\frac{1}{2}$ in order to obtain
			\begin{align*}
			 -\norm{\bg}
			&\leq\langle\overline{\bw},\bg\rangle
			=\frac{1}{4}\langle\overline{\bw},\overline{\bx}\rangle-s\left(1-\frac{1}{2}\cdot\frac{1}{2}\right)
			\leq-\frac{1}{4}\sqrt{1-16\epsilon^2}-\frac{3}{4}s
			\\&\leq-\frac{1}{4}\sqrt{1-16\epsilon^2}~.  
			\end{align*}
			This implies that $\frac{1}{4}\sqrt{1-16\epsilon^2}\leq\norm{\bg}\leq\epsilon$, which does not hold for any $\epsilon<\frac{1}{4\sqrt{2}}$.
		\end{itemize}
\end{itemize}
\end{proof}

Finally, given some nonzero $\bw\in\reals^d$ such that $w_d=0$, we are ready to consider the function
\begin{align*}
F_{\bw}(x_1,\dots,x_d):=&
\max\left\{-1,f_{\bw}(\bx-\bx^*)\right\} \\
=&\max\left\{-1,
h(x_d)+\frac{1}{4}\sqrt{\sum_{i=1}^{d-1}x_i^{2}}-\left[\langle\overline{\bw},\bx-\bx^*+\bw\rangle-\frac{1}{2}\|\bx-\bx^*+\bw\|\right]_{+}\right\}
~.
\end{align*}

\begin{lemma} \label{lem: nonsmooth finale}
The following hold:
\begin{itemize}
    \item $F_{\bw}(\cdot)$ is $\frac{15}{4}$-Lipschitz, $F_{\bw}(\mathbf{0})-\inf_{\bx}F_{\bw}(\bx)\leq2$ and $\inf\left\{\norm{\bx}~|~\partial F_{\bw}(\bx)=\{\mathbf{0}\}\right\}\leq 13$.
    \item Any $\epsilon$-stationary point $\bx$ for $\epsilon<\frac{1}{4\sqrt{2}}$ satisfies $F_{\bw}(\bx)=-1$.
    \item There exists a choice of $\bw$, such that if we run $\Acal$ on $F_{\bw}(\cdot)$, then with probability at least $1-2T\exp(-d/36)$ the algorithm's iterates $\bx_1^{F_{\bw}},\dots,\bx_T^{F_{\bw}}$ satisfy $\min_{t\in[T]}F_{\bw}(\bx_t^{F_{\bw}})>0$.
\end{itemize}
\end{lemma}
\begin{proof}
First, recall that $f_{\bw}(\cdot)$ is $\frac{15}{4}$-Lipschitz by \lemref{lem: f_w properties}. Combining this with the fact that $\bx\mapsto\bx-\bx^*,~z\mapsto\max\{-1,z\}$ are both $1$-Lipschitz yields the desired Lipschitz bound. Moreover, we see that $\inf_{\bx}F_{\bw}(\bx)\geq-1$, and by definition of $F_{\bw}(\cdot)$ and \lemref{lem: h provider}: $F_{\bw}(\mathbf{0})\leq h(0)=1$. Combining the two observations gives
\[
F_{\bw}(\mathbf{0})-\inf_{\bx}F_{\bw}(\bx)
\leq 1+1=2~.
\]
For the remaining claim in the first bullet, consider $\bv=12\overline{\bw}+\bx^*$. By \lemref{lem: h provider} we have $\norm{\bv}\leq12+1=13$, thus it is enough to show that $\bv$ is a stationary point of $F_{\bw}$. In particular, it is enough to show that $f_{\bw}(\bv-\bx^*)<-1$, since by the continuity of $f_{\bw}$ this will imply that $F_{\bw}\equiv-1$ in a neighborhood of $\bv$. Indeed, using the facts that $\bv-\bx^*=12\overline{\bw},~w_d=0$ we get
\begin{align*}
    f_{\bw}(\bv-\bx^*)
    &=h(x^*)+\frac{1}{4}\cdot12\norm{\overline{\bw}}-\left[\langle\overline{\bw},12\overline{\bw}+\bw\rangle-\frac{1}{2}\|12\overline{\bw}+\bw\|\right]_{+}
    \\
    &<h(0)+3-\frac{1}{2}\norm{12\overline{\bw}+\bw}
    <1+3-\frac{1}{2}\cdot12<-1~.
\end{align*}
As to the second bullet, suppose $\bx$ is an $\epsilon$-stationary point for some $\epsilon<\frac{1}{4\sqrt{2}}$. Namely, there exists $\bg\in\partial F_{\bw}(\bx)$ such that $\norm{\bg}\leq\epsilon$. Assume by contradiction that $F_{\bw}(\bx)>-1$. Since the set $\{\by|~F_{\bw}(\by)>-1\}$ is an open set, it follows that for all $\by$ in some neighborhood of $\bx$: $F_{\bw}(\by)>-1$. Hence, for all $\by$ in some neighborhood of $\bx$: $F_{\bw}(\by)=f_{\bw}(\by-\bx^*)$, which in particular implies that $\partial F_{\bw}(\bx)=\partial f_{\bw}(\bx-\bx^*)$. We conclude that $\bg\in\partial f_{\bw}(\bx-\bx^*)$ and satisfies $\norm{\bg}<\frac{1}{4\sqrt{2}}$. Thus $(\bx-\bx^*)$ is an $\epsilon$-stationary point of $f_{\bw}(\cdot)$ for some $\epsilon<\frac{1}{4\sqrt{2}}$, which is a contradiction to \lemref{lem: f_w properties}.

In order to prove the third bullet, we start by noticing that
\[
\bx\in \left\{\bx:\langle\overline{\bw},\overline{\bx-\bx^*+\bw}\rangle\leq\frac{1}{2}\right\}
\cup \{\bx^*-\bw\}
\implies 
\langle\overline{\bw},\bx-\bx^*+\bw\rangle-\frac{1}{2}\|\bx-\bx^*+\bw\|\leq0~,
\]
from which it follows that
\begin{equation} \label{eq: x angle condition}
\forall \bx\in \left\{\bx:\langle\overline{\bw},\overline{\bx-\bx^*+\bw}\rangle\leq\frac{1}{2}\right\}
\cup \{\bx^*-\bw\}
~:~
F_{\bw}(\bx)=\bar{f}(\bx):=h(x_d)+\frac{1}{4}\sqrt{\sum_{i=1}^{d-1}x_i^2}~.
\end{equation}
Indeed, for any such $\bx$ the ReLU term in the definition of $F_{\bw}(\cdot)$ vanishes, and the remaining function (which is non-negative) is greater than $-1$. We continue by showing that \eqref{eq: x angle condition} holds over a set of more convenient form. In order to do that, fix some $\bx$ such that $\langle\overline{\bw},\overline{\bx-\bx^*+\bw}\rangle>\frac{1}{2}$ (i.e. the \emph{opposite} condition). Multiplying by $\norm{\bx-\bx^*+\bw}$ gives
\[
\langle\overline{\bw},\bx-\bx^*\rangle+\norm{\bw}=
\langle\overline{\bw},\bx-\bx^*+\bw\rangle
>\frac{1}{2}\norm{\bx-\bx^*+\bw}
\geq \frac{1}{2}\left(\norm{\bx-\bx^*}-\norm{\bw}\right)~.
\]
For $\bx=\bx^*$ the inequality above is trivially satisfied. For $\bx\neq \bx^*$, dividing by $\norm{\bx-\bx^*}$ and rearranging yields
\[
\langle\overline{\bw},\overline{\bx-\bx^*}\rangle>\frac{1}{2}-\frac{3\norm{\bw}}{2\norm{\bx-\bx^*}}~.
\]
The contrapositive allows us to deduce that any $\bx$ which does \emph{not} satisfy the condition above belongs to $\{\bx:\langle\overline{\bw},\overline{\bx-\bx^*+\bw}\rangle\leq\frac{1}{2}\}\cup \{\bx^*-\bw\}$, so we get by \eqref{eq: x angle condition}:
\begin{equation} \label{eq: x angle simplified}
\forall \bx\neq \bx^*:
\langle\overline{\bw},\overline{\bx-\bx^*}\rangle\leq\frac{1}{2}-\frac{3\norm{\bw}}{2\norm{\bx-\bx^*}}
\implies
F_{\bw}(\bx)=\bar{f}(\bx):=h(x_d)+\frac{1}{4}\sqrt{\sum_{i=1}^{d-1}x_i^2}~. 
\end{equation}
With this equation in hand, we turn to describe how $\bw$ should be set in order to establish the third bullet. Consider a random vector $\bw\in\reals^d$ which is distributed as follows: 
\begin{equation} \label{eq: w distribution}
(w_1,\dots,w_{d-1})\sim\mathrm{Unif}\left(\frac{\rho}{99}\cdot \SSS^{d-2}\right),\ \Pr[w_d=0]=1~,
\end{equation}
where $\frac{\rho}{99}\cdot \SSS^{d-2}:=\{(y_1,\dots,y_{d-1})|\sum_{i=1}^{d-1}y_i^2=\frac{\rho}{99}\}$ is the $(d-2)$-dimensional sphere of radius $\frac{\rho}{99}$. Note that $\norm{\bw}=\frac{\rho}{99}$, which by plugging into \eqref{eq: x angle simplified} gives
\begin{equation} \label{eq: x angle simplified 2}
\forall \bx\neq \bx^*:
\langle\overline{\bw},\overline{\bx-\bx^*}\rangle\leq\frac{1}{2}-\frac{\rho}{66\norm{\bx-\bx^*}}
\implies
F_{\bw}(\bx)=\bar{f}(\bx):=h(x_d)+\frac{1}{4}\sqrt{\sum_{i=1}^{d-1}x_i^2}~. 
\end{equation}
Let $\bx_1^{\bar{f}},\dots,\bx_{T}^{\bar{f}}$ be the (possibly random) iterates produced by $\Acal$ when ran on $\bar{f}(\cdot)$. Note that if
\begin{equation} \label{eq: x_t^barf event}
    \left(\min_{t\in[T]}\norm{\bx_{t}^{\bar{f}}-\bx^*}\geq \rho>0\right)
    \land 
    \left(
    \max_{t\in[T]}\langle\overline{\bw},\overline{(\bx_t^{\bar{f}}-\bx^*)}\rangle
	<\frac{1}{3}
    \right)
\end{equation}
then for all $t\in[T]:$
    \[
	\langle\overline{\bw},\overline{(\bx_t^{\bar{f}}-\bx^*)}\rangle
	<\frac{1}{3}
	<\frac{1}{2}-\frac{\rho}{66\rho}
	\leq\frac{1}{2}-\frac{\rho}{66\norm{\bx_{t}^{\bar{f}}-\bx^*}}~,
	\]
as well as $\bx_t^{\bar{f}}\neq\bx^*$. Thus, by \eqref{eq: x angle simplified 2}, this means that \eqref{eq: x_t^barf event} implies that $F_{\bw}(\bx_t^{\bar{f}})=\bar{f}(\bx_t^{\bar{f}})$ for all $t\in[T]$. Moreover, using the fact that $\bx_t^{\bar{f}}$ is bounded away from $\bx^*$, it is easily verified that the condition in \eqref{eq: x angle simplified 2} also holds for
all $\bx$ in some neighborhood of $\bx_t^{\bar{f}}$, so actually $F_{\bw}(\cdot)$ is identical to $\bar{f}(\cdot)$ on these neighborhoods, implying the same local oracle response. Hence, assuming the event in \eqref{eq: x_t^barf event} occurs, if we run the algorithm on $F_{\bw}(\cdot)$ rather than $\bar{f}(\cdot)$, then the produced iterates $\bx_{1}^{F_{\bw}},\dots,\bx_{T}^{F_{\bw}}$ are identical to $\bx_{1}^{\bar{f}},\dots,\bx_{T}^{\bar{f}}$. That being the case, we would get
	\[
	\min_{t\in[T]}F_{\bw}(\bx_{t}^{F_{\bw}})
	=\min_{t\in[T]}F_{\bw}(\bx_{t}^{\bar{f}})
	=\min_{t\in[T]}\bar{f}(\bx_{t}^{\bar{f}})
	>0~,
	\]
	where the last inequality utilizes the fact that $\norm{\bx_t^{\bar{f}}-\bx^*}>0$. Overall we see that
	\[
	\left(\min_{t\in[T]}\norm{\bx_{t}^{\bar{f}}-\bx^*}\geq \rho\right)
	\land 
    \left(
    \max_{t\in[T]}\langle\overline{\bw},\overline{(\bx_t^{\bar{f}}-\bx^*)}\rangle
	<\frac{1}{3}
    \right)
	\implies
	\min_{t\in[T]}F_{\bw}(\bx_{t}^{F_{\bw}})>0~.
	\]
Thus, in order to finish the proof, it is enough to show that there exists $\bw$, such that
\begin{equation} \label{eq: enough to show}
\Pr_{\Acal}\left[
\left(\min_{t\in[T]}\norm{\bx_{t}^{\bar{f}}-\bx^*}\geq \rho\right)
\land 
\left(
\max_{t\in[T]}\langle\overline{\bw},\overline{(\bx_t^{\bar{f}}-\bx^*)}\rangle
<\frac{1}{3}
\right)
\right]\geq1-2T\exp(-d/36)~.
\end{equation}
In order to prove this claim, we observe that:
\begin{enumerate}
    \item By \lemref{lem: h provider} we know that
    $\Pr_{\Acal}[\min_{t\in[T]}\norm{\bx_t^{\bar{f}}-\bx^*}\geq \rho]\geq1-\delta=1-T\exp(-d/36)$.
    \item If we fix some vectors $\bu_1,\ldots,\bu_T$ in $\reals^{d-1}$ such that $\forall t:\norm{\bu_{t}}\leq1$, and pick a unit vector $\bu\in\reals^{d-1}$ uniformly at random, then by a union bound and a standard concentration of measure on the sphere argument (e.g., \citealp[Lemma 2.2]{ball1997elementary}), $\Pr(\max_t \langle\bu,\bu_t\rangle\geq\alpha)\leq T\cdot\Pr(\langle\bu,\bu_1\rangle\geq\alpha)\leq T\exp(-(d-1)\alpha^2/2)$. For any realization of $\Acal$'s randomness such that such that for all $t\in[T]:\norm{\bx_t^{\bar{f}}-\bx^*}\geq \rho$, by setting
    \[
    \alpha=1/3,\ \ \bu=\overline{(w_1,\dots,w_{d-1})},\ \  \bu_{t}=\frac{1}{\norm{\bx_t^{\bar{f}}-\bx^*}}((\bx_t^{\bar{f}}-\bx^*)_1,\dots,(\bx_t^{\bar{f}}-\bx^*)_{d-1})~,
    \]
    while noticing that $\langle\bu,\bu_t\rangle=\langle\overline{\bw},\overline{(\bx_t^{\bar{f}}-\bx^*)}\rangle$ since $w_{d}=0$, we get
    \[
    \Pr_{\bw}\left[\max_{t\in[T]} \langle\overline{\bw},\overline{(\bx_t^{\bar{f}}-\bx^*)}\rangle\geq 1/3\right]\leq T\exp(-d/36)
    ~.\]
\end{enumerate}
Combining the two observations in a formal manner results in
\begin{equation} \label{eq: nested prob bound}
    \Pr_{\Acal}\left[\Pr_{\bw}\left[
    E_{\Acal,\bw}
    \middle|\Acal\right]\geq1-T\exp(-d/36)\right]
    \geq1-T\exp(-d/36)~,
\end{equation}
where
\[
    E_{\Acal,\bw}:=\left(\min_{t\in[T]}\norm{\bx_{t}^{\bar{f}}-\bx^*}    \geq \rho\right)
    \land 
    \left(
    \max_{t\in[T]}\langle\overline{\bw},\overline{(\bx_t^{\bar{f}}-\bx^*)}\rangle
    <\frac{1}{3}
    \right)~.
\]
Finally, using the law of total expectation and \eqref{eq: nested prob bound} we get
\begin{align*}
    &\Pr_{\Acal,\bw}[E_{\Acal,\bw}]
    =\E_{\Acal}[\Pr_{\bw}[E_{\Acal,\bw}|\Acal]]
    \\
     & \geq\E_{\Acal}\left[\Pr_{\bw}\left[E_{\Acal,\bw}|\Acal:\Pr_{\bw}[E_{\Acal,\bw}|\Acal]\geq1-T\exp(-d/36)\right]
    \cdot\Pr_{\Acal}\left[\Pr_{\bw}[E_{\Acal,\bw}|\Acal]\geq1-T\exp(-d/36)\right]\right]
    \\
    &\geq\E_{\Acal}\left[(1-T\exp(-d/36)\cdot(1-T\exp(-d/36)\right]
    \\
    &\geq(1-T\exp(-d/36))^2
    \\
    &\geq1-2T\exp(-d/36)~.
\end{align*}
Consequently, by the probabilistic method, there exists some fixed choice of $\bw$ such that
\[
\Pr_{\Acal}[E_{\Acal,\bw}]\geq1-2T\exp(-d/36)~,
\]
which is exactly \eqref{eq: enough to show}, finishing the proof.
\end{proof}

The theorem is an immediate corollary of the previous lemma: With the specified high probability, $\min_t F_{\bw}(\bx_t)>0$, even though all $\epsilon$-stationary points (for any $\epsilon<\frac{1}{4\sqrt{2}}$) have a value of $-1$. Since $F_{\bw}$ is also $\frac{15}{4}$-Lipschitz, we get that the distance of any $\bx_t$ from an $\epsilon$-stationary point must be at least $\frac{0-(-1)}{\frac{15}{4}}=\frac{4}{15}$. Simplifying the numerical terms by choosing a large enough constant $C$ and a small enough constant $c$, and relabeling $F_{\bw}$ as $f$, the theorem follows.

\subsection{Proof of \thmref{thm:smoothingmain}}

We start the proof by showing that without loss of generality we can impose certain assumptions on the parameters of interest. First, if $\epsilon\geq1$ then the right hand side of \eqref{eq:theorem inequality} is negative for any $c_2<1$, which makes the theorem trivial. Consequently, we can assume $\epsilon<1$. Using \lemref{lemma: L>1/eps} in Appendix \ref{app:technical lemmas}, this also implies that $L\geq\frac{1}{8}$ since otherwise an $L$-smooth $\epsilon$-approximation does not exist in the first place in case of $1$-Lipschitz function $\bx\mapsto |x_1|$ (in particular, no such smoother exists).
Therefore, if $\sqrt{\log\left(\left(M+1\right)\left(T+1\right)\right)}\geq\frac{\sqrt{d}}{32r}$ then
\[
L\sqrt{\log\left(\left(M+1\right)\left(T+1\right)\right)}\geq\frac{1}{8}\cdot\frac{\sqrt{d}}{32r}>\frac{1}{256}\cdot\frac{\sqrt{d}}{r}\left(1-\epsilon\right)~,
\]
which proves the theorem. Thus we can assume throughout the proof that 
\begin{equation} \label{eq: 1/16B...reduction}
    \sqrt{\log\left(\left(M+1\right)\left(T+1\right)\right)}<\frac{\sqrt{d}}{32r}\implies\frac{1}{16r}\sqrt{\frac{d}{\log\left(\left(M+1\right)\left(T+1\right)\right)}}>2~.
\end{equation}

Our strategy is to define a distribution over a family of ``hard'' 1-Lipschitz functions over $\reals^d$, for which we will show that \eqref{eq:theorem inequality} must hold for some function supported by this distribution. 
Before we turn to do so, we will show that when a smoother which satisfies $\mathsf{TICF}$ acts on a constant function, it returns a gradient estimate of small norm. This will be crucial later on, since we will construct a function which looks ``locally constant'' at many points of interest, thus deceiving the smoother.

\begin{lemma} \label{lemma: constant func norm<eps}
    If $\Scal$ is an $(L,\epsilon,T,M,r)$-smoother satisfying $\mathsf{TICF}$, then for any constant function $f$ and any $\bx\in\reals^d:\ \left\|\E\left[\Scal\left(f,\bx\right)\right]\right\|\leq\epsilon$.
\end{lemma}
\begin{proof}
    Denote $\bv:=\E\left[\Scal\left(f,\bx\right)\right]$, and note that by the $\mathsf{TICF}$ property $\bv$ does not depend on $\bx$. Let $\tilde{f}$ be the $\epsilon$-approximation of $f$ implicitly computed by $\Scal$, then by the definition of a smoothing algorithm, we have for all $\bx\in\reals^d$:
    \begin{gather*}
        \left\|\bv-\nabla\tilde{f}\left(\bx\right)\right\|\leq\epsilon \\
        \implies 
        \|\bv\|^2-\left\langle\nabla\tilde{f}\left(\bx\right),\bv\right\rangle=
        \left\langle\bv-\nabla\tilde{f}\left(\bx\right),\bv\right\rangle
        \leq\left\|\bv-\nabla\tilde{f}\left(\bx\right)\right\|\cdot\left\|\bv\right\|
        \leq\epsilon\|\bv\| \\
        \implies \left\langle\nabla\tilde{f}\left(\bx\right),\bv\right\rangle\geq\|\bv\|^2-\epsilon\|\bv\|~.
    \end{gather*}
    Define the one dimensional projected function $\tilde{f}_{\bv}(t):=\tilde{f}(t\cdot\bv)$. Then for all $t\geq 0$,
    \begin{align} \label{eq:ftilde(t)-ftilde(0)}
    \tilde{f}_{\bv}\left(t\right)-\tilde{f}_{\bv}\left(0\right)
    &=\int_{0}^{t}\tilde{f}_{\bv}'(z)dz
    =\int_{0}^{t}\left\langle\nabla\tilde{f}\left(z\cdot\bv\right),\bv\right\rangle dz \nonumber \\
    &\geq\int_{0}^{t}\left(\|\bv\|^2-\epsilon\|\bv\|\right)dz=t\left(\|\bv\|^2-\epsilon\|\bv\|\right)
    =t\|\bv\|\left(\|\bv\|-\epsilon\right)~.
    \end{align}
    On the other hand, $\tilde{f}_{\bv}(t),\tilde{f}_{\bv}(0)$ are both $\epsilon$-approximations of the same constant, since $f$ is a constant function. Thus, $|\tilde{f}_{\bv}(t)-\tilde{f}_{\bv}(0)|\leq2\epsilon$. Combining this with \eqref{eq:ftilde(t)-ftilde(0)} yields for all $t\geq 0:~2\epsilon~\geq~ t\|\bv\|\left(\|\bv\|-\epsilon\right)$.
    This can hold for all $t\geq 0$ only if $(\norm{\bv}-\epsilon) \leq 0$, implying the lemma.
\end{proof}

We are now ready to define a family of functions, with the rest of the proof devoted to analyze how a smoother acts on them.
Relying on \eqref{eq: 1/16B...reduction} we can define the set
\[
\Delta:=\left\{16r\sqrt{\frac{\log \left(\left(M+1\right)\left(T+1\right)\right)}{d}}\cdot k ~\middle|~ k=0,1,\ldots,\left\lfloor\frac{1}{16r}\sqrt{\frac{d}{\log \left(\left(M+1\right)\left(T+1\right)\right)}}\right\rfloor\right\}~.
\]
That is, a grid on $[0,1]$ which consists of points of distance $16r\sqrt{\frac{\log \left(\left(M+1\right)\left(T+1\right)\right)}{d}}$ one from another. We further define the ``inflation'' of $\Delta$ by $4r\sqrt{\frac{\log \left(\left(M+1\right)\left(T+1\right)\right)}{d}}$ around every point:\footnote{Note we use the quantities $T+1,M+1$ instead of the seemingly more natural $T,M$, since otherwise the logarithmic term in \eqref{eq:theorem inequality} can vanish, resulting in an invalid theorem. This would have occurred for randomized smoothing, where $T=M=1$.}
\[
\overline{\Delta}:=\left\{x\in\reals~\middle|~\exists p\in\Delta:|p-x|\leq4r\sqrt{\frac{\log \left(\left(M+1\right)\left(T+1\right)\right)}{d}}\right\}~.
\]
Now we define the function $g:\reals\to\reals$ as the unique continuous function which satisfies (see \figref{fig:g illustration} for an illustration)
\begin{align*}
    g(0)&=0 \\
    g'\left(x\right)&=
    \begin{cases}
    1~, & x\notin\overline{\Delta}\\
    0~, & x\in\overline{\Delta}\\
    \end{cases}
\end{align*}
Finally, we are ready to consider
\[
f_{\bw}\left(\bx\right)=g\left(\langle\bx,\bw\rangle\right)~,
\]
where $\bw\in\SSS^{d-1}$ is drawn uniformly from the unit sphere. The distribution over $\bw$ specifies a distribution over the functions $f_{\bw}$. We start by claiming that these functions are indeed in our function class of interest:
\begin{lemma}
    For all $\bw\in\SSS^{d-1}$, $f_{\bw}(\cdot)$ is 1-Lipschitz.
\end{lemma}
\begin{proof}
    It is clear by construction that $g$ is 1-Lipschitz. Thus
    \[
    \left|f\left(\bx\right)-f\left(\by\right)\right|
    =
    \left|g\left(\left\langle\bx,\bw\right\rangle \right)-g\left(\left\langle\by,\bw\right\rangle \right)\right|
    \leq
    \left|\left\langle\bx,\bw\right\rangle -\left\langle \by,\bw\right\rangle \right|
    =
    \left|\left\langle\bx-\by,\bw\right\rangle \right|
    \leq
    \left\Vert \bx-\by\right\Vert ~.
    \]
\end{proof}
The following lemma is the key lemma of the proof. It will show that there exists a function supported by the distribution we defined, such that many points mislead the smoother by appearing as if the function is constant - hence, the the smoother returns gradient estimates of small norm. Formally:

\begin{lemma} \label{lemma:distinguish f constant}
There exists $\bw\in\SSS^{d-1}$ such that for all $\delta\in\Delta:\ \E_{\xi}\left[\left\|\Scal\left(f_{\bw},\delta\bw\right)\right\|\right]\leq\epsilon+\frac{1}{32}$.
\end{lemma}
\begin{proof}
    Let $\bx^{(\bw)}_1,\dots,\bx^{(\bw)}_T$ be the (possibly randomized) queries produced by $\Scal\left(f_{\bw},\bf{0}\right)$.
    Consider the event $E_{\bw}$, in which for all $i\in[T]:\ \left|\langle \bx^{(\bw)}_i,\bw\rangle\right|<4r\sqrt{\frac{\log \left(\left(M+1\right)\left(T+1\right)\right)}{d}}$. Note that if $E_{\bw}$ occurs then for all $\delta\in\Delta,i\in[T],\bv\in\reals^d$:
    \begin{equation} \label{eq:f_w=g(delta+xw+vw)}
        f_{\bw}\left(\bx^{(\bw)}_i+\delta\bw+\bv\right)
       =g\left(\left\langle \bx^{(\bw)}_i+\delta\bw+\bv,\bw\right\rangle\right)
       =g\left(\delta+\left\langle \bx^{(\bw)}_i,\bw\right\rangle+\left\langle\bv,\bw\right\rangle\right)~.
    \end{equation}
    In particular, as long as $\left\|\bv\right\|<4r\sqrt{\frac{\log\left(\left(M+1\right)\left(T+1\right)\right)}{d}}-\left|\left\langle \bx_{i}^{(\bw)},\bw\right\rangle \right|$, which by Cauchy-Schwarz implies
    \begin{equation*}
        \left|\left\langle \bx_{i}^{(\bw)},\bw\right\rangle +\left\langle \bv,\bw\right\rangle\right| 
        <4r\sqrt{\frac{\log\left(\left(M+1\right)\left(T+1\right)\right)}{d}}~,
    \end{equation*}
    we get by construction of $g$ and \eqref{eq:f_w=g(delta+xw+vw)} that
    \begin{equation*}
        f_{\bw}\left(\bx^{(\bw)}_i+\delta\bw+\bv\right)=g\left(\delta\right)~.
    \end{equation*}
    In other words, if $E_\bw$ occurs then inside some neighborhood of $\bx^{(\bw)}_i+\delta\bw$, the function $f_{\bw}$ is identical to the constant function $g\left(\delta\right)$. Therefore, if $E_\bw$ occurs the local oracle $\OO$ satisfies for any $\delta\in\Delta,i\in[T]$:
    \begin{equation} \label{eq:Oracle= if Ew}
        \OO_{f_{\bw}}\left(\bx_{i}^{(\bw)}+\delta\bw\right)
        =\OO_{\bx\mapsto g\left(\delta\right)}\left(\bx_{i}^{(\bw)}+\delta\bw\right)~.
    \end{equation}
    Fix some $\delta_{0}\in\Delta$, and let $\tilde{\bx}^{(\bw)}_1,\dots,\tilde{\bx}^{(\bw)}_T$ be the (possibly randomized) queries produced by $\Scal\left(f_{\bw},\delta_{0}\bf{\bw}\right)$.
    We will now show that conditioned on $E_{\bw}$, for all $i\in[T]$:
    \begin{equation} \label{eq: tilde q_i=q_i+delta}
        \tilde{\bx}_{i}^{(\bw)}=\bx_{i}^{(\bw)}+\delta_{0}{\bw}~,
    \end{equation}
    in the sense that for every realization of $\Scal'$'s randomness $\xi$ they are equal. We show this by induction on $i$. For $i=1$, using $\mathsf{TICF}$:
    \[
    \tilde{\bx}_{1}^{(\bw)}=\Scal^{(1)}\left(\xi,\delta_{0}{\bf\bw}\right)=\Scal^{(1)}\left(\xi,{\bf 0}\right)+\delta_{0}{\bf\bw}=\bx_{1}^{(\bw)}+\delta_{0}{\bf\bw}~.
    \]
    Assuming this is true up until $i$, then by the induction hypothesis, \eqref{eq:Oracle= if Ew} and $\mathsf{TICF}$:
    \begin{align*}
        \tilde{\bx}_{i+1}^{(\bw)}&=\Scal^{(i)}\left(\xi,\delta_{0}{\bw},\OO_{f_\bw}\left(\tilde{\bx}_{1}^{(\bw)}\right),\dots,\OO_{f_\bw}\left(\tilde{\bx}_{i}^{(\bw)}\right)\right)\\
        &=\Scal^{(i)}\left(\xi,\delta_{0}{\bw},\OO_{f_\bw}\left(\bx_{1}^{(\bw)}+\delta_{0}{\bw}\right),\dots,\OO_{f_\bw}\left(\bx_{i}^{(\bw)}+\delta_{0}{\bw}\right)\right)\\
        &=\Scal^{(i)}\left(\xi,\delta_{0}{\bw},\OO_{\bx\mapsto g\left(\delta_{0}\right)}\left(\bx_{1}^{(\bw)}+\delta_{0}{\bw}\right),\dots,\OO_{\bx\mapsto g\left(\delta_{0}\right)}\left(\bx_{i}^{(\bw)}+\delta_{0}{\bw}\right)\right)\\
        &=\Scal^{(i)}\left(\xi,\mathbf{0},\OO_{\bx\mapsto g\left(0\right)}\left(\bx_{1}^{(\bw)}\right),\dots,\OO_{\bx\mapsto g\left(0\right)}\left(\bx_{1}^{(\bw)}\right)\right)+\delta_{0}\bw\\
        &=\Scal^{(i)}\left(\xi,\mathbf{0},\OO_{f_\bw}\left(\bx_{1}^{(\bw)}\right),\dots,\OO_{f_\bw}\left(\bx_{i}^{(\bw)}\right)\right)+\delta_{0}\bw\\
        &=\bx_{i+1}^{(\bw)}+\delta_{0}\bw~.
    \end{align*}
    Having established \eqref{eq: tilde q_i=q_i+delta} for any $\delta\in\Delta$, we turn to show that for all $\delta\in\Delta$:
    \begin{equation} \label{eq: distinguish from constant if E_w}
        \E_{\xi}\left[\Scal\left(f_{\bw},\delta{\bf\bw}\right)\middle|E_{\bw}\right]
        =\E_{\xi}\left[\Scal\left(\bx\mapsto0,{\bf 0}\right)\middle|E_{\bw}\right]~.
    \end{equation}
    Indeed, by \eqref{eq: tilde q_i=q_i+delta}, \eqref{eq:Oracle= if Ew} and $\mathsf{TICF}$:
    \begin{align*}
    \mathbb{E}_{\xi}\left[\Scal\left(f_{\bw},\delta{\bf\bw}\right)\middle|E_{\bw}\right]
    &=\mathbb{E}_{\xi}\left[\Scal^{(out)}\left(\xi,\delta{\bf\bw},\OO_{f_{\bw}}\left(\tilde{\bx}_{1}^{(\bw)}\right),\dots,\OO_{f_{\bw}}\left(\tilde{\bx}_{T}^{(\bw)}\right)\right)\middle|E_{\bw}\right] \\
    &=\mathbb{E}_{\xi}\left[\Scal^{(out)}\left(\xi,\delta{\bf\bw},\OO_{f_{\bw}}\left(\bx_{1}^{(\bw)}+\delta\bw\right),\dots,\OO_{f_{\bw}}\left(\bx_{T}^{(\bw)}+\delta\bw\right)\right)\middle|E_{\bw}\right] \\
    &=\mathbb{E}_{\xi}\left[\Scal^{(out)}\left(\xi,\delta{\bf\bw},\OO_{\bx\mapsto g\left(\delta\right)}\left(\bx_{1}^{(\bw)}+\delta\bw\right),\dots,\OO_{\bx\mapsto g\left(\delta\right)}\left(\bx_{T}^{(\bw)}+\delta\bw\right)\right)\middle|E_{\bw}\right] \\
    &=\mathbb{E}_{\xi}\left[\Scal^{(out)}\left(\xi,\mathbf{0},\OO_{\bx\mapsto0}\left(\bx_{1}^{(\bw)}\right),\dots,\OO_{\bx\mapsto0}\left(\bx_{T}^{(\bw)}\right)\right)\middle|E_{\bw}\right] \\
    &=\mathbb{E}_{\xi}\left[\Scal\left(\bx\mapsto0,{\bf 0}\right)\middle|E_{\bw}\right]~.
    \end{align*}
    We now turn to show that $E_{\bw}$ is likely to occur. Fix some realization of $\Scal$'s randomness $\xi$, and let $\bq^{\xi}_1,\dots,\bq^{\xi}_T$ be the (deterministic) queries produced by $\Scal\left(\by\mapsto0,\bf{0}\right)$. We claim that if for all $i\in[T]:\ \left|\langle \bq^{\xi}_i,\bw\rangle\right|<4r\sqrt{\frac{\log\left(\left(M+1\right)\left(T+1\right)\right)}{d}}$ then $\left(\bq^{\xi}_1,\dots,\bq^{\xi}_T\right)=\left(\bx^{(\bw)}_1,\dots,\bx^{(\bw)}_T\right)$ independently of $\bw$. We show this by induction on $i$. For $i=1$:
    \[
    \bq_1^{\xi}=
    \Scal^{(1)}\left(\xi,\bf{0}\right)=\bx_1^{(\bw)}~.
    \]
    Assuming true up until $i$, then
    \begin{align*}
        \bq_{i+1}^{\xi} 
        &=\Scal^{(i)}\left(\xi,{\bf 0},\OO_{\bx\mapsto0}\left(\bq_1^{\xi}\right),\dots,\OO_{\bx\mapsto0}\left(\bq_i^{\xi}\right)\right)\\
        &=\Scal^{(i)}\left(\xi,{\bf 0},\OO_{f_\bw}\left(\bq_1^{\xi}\right),\dots,\OO_{f_\bw}\left(\bq_i^{\xi}\right)\right)\\
        &=\Scal^{(i)}\left(\xi,{\bf 0},\OO_{f_\bw}\left(\bx_{1}^{(\bw)}\right),\dots,\OO_{f_\bw}\left(\bx_{i}^{(\bw)}\right)\right)\\
        &=\bx_{i+1}^{(\bw)}~,
    \end{align*}
    where we used the assumption on $\bq_i^{\xi}$ and the induction hypothesis.
    Recall that by assumption on the algorithm $\left\|\bq^{\mathcal{\xi}}_i\right\|\leq{r}$ for all $i\in[T]$. Using the union bound and concentration of measure on the sphere (e.g., \citealp[Lemma 2.2]{ball1997elementary}) we can bound the probability of the complementary event
    \begin{align*}
    \Pr_{\bw}\left[E_{\bw}^c~\middle|~\xi\right]
    =
    &\Pr_{\bw}\left[\exists i\in[T]:\ \left|\langle \bq^{\xi}_i,\bw\rangle\right|\geq4r\sqrt{\frac{\log\left(\left(M+1\right)\left(T+1\right)\right)}{d}}\right] \\
    =&
    \Pr_{\bw}\left[\exists i\in[T]:\ \left|\left\langle \frac{1}{r}\bq^{\xi}_i,\bw\right\rangle\right|\geq4\sqrt{\frac{\log\left(\left(M+1\right)\left(T+1\right)\right)}{d}}\right] \\
    \leq&
    T\cdot 2\exp\left(-\frac{d\cdot\left(4\sqrt{\frac{\log\left(\left(M+1\right)\left(T+1\right)\right)}{d}}\right)^2}{2}\right)\\
    =&\frac{2T}{\left(M+1\right)^8\left(T+1\right)^8}
    \leq\frac{2}{\left(M+1\right)^8\left(T+1\right)^7}~.
    \end{align*}
    This inequality holds for any realization of $\Scal$'s randomness $\xi$, hence by the law of total probability
    \[
    \Pr_{\xi,\bw}\left[E_{\bw}^c\right]\leq\frac{2}{\left(M+1\right)^8\left(T+1\right)^7}~.
    \]
    In particular, since $\Pr_{\xi,{\bw}}\left[E_{\bw}^c\right]=\E_{\bw}\left[\Pr_{\xi}\left[E_{\bw}^c\middle|\bw\right]\right]$, there exists $\bw\in{\SSS^{d-1}}$ such that 
    \begin{equation} \label{eq:P_A[notE_w]<2/MQ}
    \Pr_{\xi}\left[ E_{\bw}^c\right]\leq\frac{2}{\left(M+1\right)^8\left(T+1\right)^7}~.
    \end{equation}
    For this fixed $\bw$, we have for all $\delta\in\Delta$ by the law of total expectation and the triangle inequality:
    \begin{equation} \label{eq:(*)+(**)}
    \left\|\E_{\xi}\left[\Scal\left(f_{\bw},\delta\bw\right)\right]\right\|
    \leq
    \left\|\underset{(*)}{\underbrace{\E_{\xi}\left[\Scal\left(f_{\bw},\delta\bw\right)\middle| E_\bw\right]\cdot\Pr_{\xi}\left[E_{\bw}\right]}}\right\|
    +
    \left\|\underset{(**)}{\underbrace{\E_{\xi}\left[\Scal\left(f_{\bw},\delta\bw\right)\middle| E_{\bw}^c\right]\cdot\Pr_{\xi}\left[ E_{\bw}^c\right]}}\right\|~.
    \end{equation}
    On one hand, by \eqref{eq: distinguish from constant if E_w}:
    \[
    (*)=\E_{\xi}\left[\Scal\left(\bx\mapsto0,\bf{0}\right)\middle|E_{\bw}\right]\cdot\Pr_{\xi}\left[E_{\bw}\right]
    =
    \E_{\xi}\left[\Scal\left(\bx\mapsto0,\bf{0}\right)\right]-\E_{\xi}\left[\Scal\left(\bx\mapsto0,\bf{0}\right)\middle| E_{\bw}^c\right]\cdot\Pr_{\xi}\left[E_{\bw}^c\right]
    ~.
    \]
    Using \lemref{lemma: constant func norm<eps}, and by incorporating the definition of $M$ in \eqref{eq:M almost-surely condition} and \eqref{eq:P_A[notE_w]<2/MQ} we get
    \begin{equation} \label{eq:|*|}
    \|(*)\|\leq
    \epsilon+
    M\cdot\frac{2}{\left(M+1\right)^8\left(T+1\right)^7}
    \leq\epsilon+\frac{2}{\left(M+1\right)^7\left(T+1\right)^7}~.
    \end{equation}
    On the other hand, by \eqref{eq:M almost-surely condition} and \eqref{eq:P_A[notE_w]<2/MQ} again we have
    \begin{equation} \label{eq:|**|}
    \|(**)\|\leq
    \left\|\E_{\xi}\left[\Scal\left(f_{\bw},\delta\bw\right)\middle| E_{\bw}^c\right]\right\|\cdot\Pr_{\xi}\left[E_{\bw}^c\right]
    \leq
    M\cdot\frac{2}{\left(M+1\right)^8\left(T+1\right)^7}
    \leq\frac{2}{\left(M+1\right)^7\left(T+1\right)^7}~.
    \end{equation}
    Overall, plugging \eqref{eq:|*|} and \eqref{eq:|**|} into \eqref{eq:(*)+(**)}, gives
    \[
    \left\|\E_{\xi}\left[\Scal\left(f_{\bw},\delta\bw\right)\right]\right\|
    \leq\epsilon+\frac{4}{\left(M+1\right)^7\left(T+1\right)^7}
    \leq\epsilon+\frac{1}{32}~,
    \]
    where the last inequality simply follows from the fact that $M>0,~T\geq1$.
\end{proof}
From now on, we fix $\bw\in\SSS^{d-1}$ which is given by the previous lemma and denote $f=f_{\bw}$. Denote by $\tilde{f}$ the $\epsilon$-approximation of $f$ with $L$-Lipschitz gradients implicitly computed by $\Scal$. We turn our focus to the directional projection:
\begin{align*}
    &\varphi:\left[0,1\right]\to\reals\\
    &\varphi(t)=\tilde{f}\left(t\cdot\bw\right)~.
\end{align*}
Note that by assumption on $\tilde{f}$, $\varphi$ is differentiable, and $\varphi'$ is $L$-Lipschitz. \lemref{lemma:distinguish f constant} ensures us that $\varphi'$ is relatively close to zero on the grid $\Delta$, as showed in the following lemma.
\begin{lemma} \label{lemma:phi'(delta)<eps+1/32}
    $\forall\delta\in\Delta:\ \left|\varphi'\left(\delta\right)\right|\leq2\epsilon+\frac{1}{32}$
\end{lemma}
\begin{proof}
    By Cauchy-Schwarz, \lemref{lemma:distinguish f constant} and the definition of a smoother, we get that for all $\delta\in\Delta$:
    \begin{align*}
    \left|\varphi'\left(\delta\right)\right|
    &=
    \left|\left\langle\nabla\tilde{f}\left(\delta\bw\right),\bw\right\rangle\right|
    \leq\left\|\nabla\tilde{f}\left(\delta\bw\right)\right\|\cdot\left\|\bw\right\|
    =\left\|\nabla\tilde{f}\left(\delta\bw\right)\right\|
    \\
    &\leq
    \left\|\E\left[\Scal\left(f,\delta\bw\right)\right]-\nabla\tilde{f}\left(\delta\bw\right)\right\|+\left\|\E\left[\Scal\left(f,\delta\bw\right)\right]\right\|
    \leq\epsilon+\epsilon+\frac{1}{32}~.
    \end{align*}
\end{proof}
By combining the fact that $\varphi'$ has small values along the grid $\Delta$, with the fact that $\varphi'$ is $L$-Lipschitz, we can bound the oscillation of $\varphi$ along the unit interval.
\begin{lemma} \label{lemma:ph(1)-ph(0)}
    $\left|\varphi\left(1\right)-\varphi\left(0\right)\right|\leq2\epsilon+\frac{1}{32}+\frac{4Lr\sqrt{\log\left(\left(M+1\right)\left(T+1\right)\right)}}{\sqrt{d}}$ .
\end{lemma}
\begin{proof}
    Denote $\delta_i=16r\sqrt{\frac{\log \left(\left(M+1\right)\left(T+1\right)\right)}{d}}\cdot{i}$, and note that for all $i\in\left[\left\lfloor\frac{1}{16r}\sqrt{\frac{d}{\log \left(\left(M+1\right)\left(T+1\right)\right)}}\right\rfloor\right]:\ \delta_i\in\Delta$. Then
    \begin{gather} 
    \left|\varphi\left(1\right)-\varphi\left(0\right)\right|
    =
    \left|\int_{0}^{1}\varphi'\left(t\right)dt\right|
    \leq
    \int_{0}^{1}\left|\varphi'\left(t\right)\right|dt
    =
    \sum_{i=0}^{\left\lfloor\frac{1}{16r}\sqrt{\frac{d}{\log \left(\left(M+1\right)\left(T+1\right)\right)}}\right\rfloor-1}\int_{\delta_i}^{\delta_{i+1}}\left|\varphi'\left(t\right)\right|dt \nonumber \\
    \leq\left(\frac{1}{16r}\sqrt{\frac{d}{\log \left(\left(M+1\right)\left(T+1\right)\right)}}\right)\cdot\max_{i}\int_{\delta_i}^{\delta_{i+1}}\left|\varphi'\left(t\right)\right|dt~. \label{eq:phi(1)-phi(0)}
    \end{gather}
    By \lemref{lemma:phi'(delta)<eps+1/32} we have $\left|\varphi'\left(\delta_i\right)\right|,\left|\varphi'\left(\delta_{i+1}\right)\right|\leq2\epsilon+\frac{1}{32}$. Recall that $\varphi'$ is $L$-Lipschitz, so $\left|\varphi'\left(t\right)\right|$ is majorized on the interval $\left[\delta_i,\delta_{i+1}\right]$ by the piecewise linear function (see \figref{fig:l illustration})
    \[
    l\left(t\right)=\begin{cases}
    2\epsilon+\frac{1}{32}+L\left(t-\delta_{i}\right) & \delta_{i}\leq t\leq\frac{\delta_{i}+\delta_{i+1}}{2}\\
    2\epsilon+\frac{1}{32}+L\left(\delta_{i+1}-t\right) & \frac{\delta_{i}+\delta_{i+1}}{2}<t\leq \delta_{i+1}
    \end{cases}
    ~.
    \]
    \begin{figure}[t]
    \centering
    \includegraphics[scale=0.35, trim=220 200 100 250]{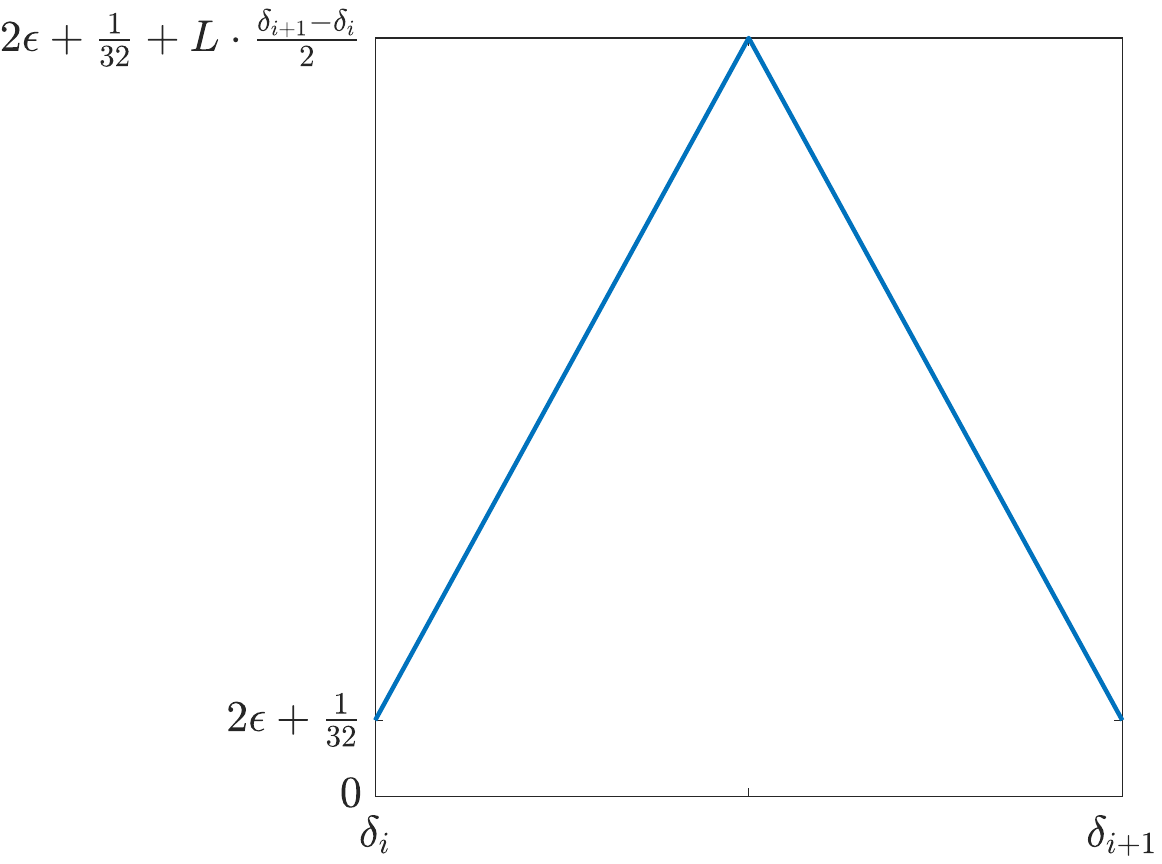}
    \caption{Illustration of $l(t)$}
    \label{fig:l illustration}
\end{figure}
Consequently,
\begin{gather}
\int_{\delta_i}^{\delta_{i+1}}\left|\varphi'\left(t\right)\right|dt
\leq
\int_{\delta_i}^{\delta_{i+1}}l\left(t\right)dt \nonumber\\ 
=
\left(2\epsilon+\frac{1}{32}\right)\cdot16r\sqrt{\frac{\log \left(\left(M+1\right)\left(T+1\right)\right)}{d}}
+L\left(8r\sqrt{\frac{\log \left(\left(M+1\right)\left(T+1\right)\right)}{d}}\right)^{2}~, \label{eq: linear integral bound}
\end{gather}
where the last equality is a direct calculation. Plugging \eqref{eq: linear integral bound}
into \eqref{eq:phi(1)-phi(0)}, we get that
\begin{equation*}
\left|\varphi\left(1\right)-\varphi\left(0\right)\right|
\leq
2\epsilon+\frac{1}{32}+\frac{4Lr\sqrt{\log\left(\left(M+1\right)\left(T+1\right)\right)}}{\sqrt{d}}~.
\end{equation*}
\end{proof}
We are now ready to finish the proof. Notice that $\varphi\left(0\right)=\tilde{f}\left(0\right),\ 
\varphi\left(1\right)=\tilde{f}\left(\bw\right)
$. Additionally, a direct calculation shows that $f(\mathbf{0})=0,\ f(\bw)\geq\frac{1}{2}$. Using the fact that $\|\tilde{f}-f\|_{\infty}\leq\epsilon$, \lemref{lemma:ph(1)-ph(0)} reveals
\begin{align*}
\frac{1}{2}&\leq\left|f(\bw)-f(0)\right|
\leq\left|\tilde{f}(\bw)-\tilde{f}(0)\right|+2\epsilon
=\left|\varphi\left(1\right)-\varphi\left(0\right)\right|+2\epsilon\\
&\leq4\epsilon+\frac{1}{32}+\frac{4Lr\sqrt{\log\left(\left(M+1\right)\left(T+1\right)\right)}}{\sqrt{d}}
\\
&\implies
L\sqrt{\log{\left(\left(M+1\right)\left(T+1\right)\right)}}\geq\frac{\sqrt{d}}{r}\left(\frac{15}{128}-\epsilon\right)~.
\end{align*}

\section{Discussion} \label{sec:discussion}

In this paper, we studied the problem of nonconvex, nonsmooth optimization from an oracle complexity perspective, and provided two main results: One (in \secref{sec:near-approximate}) is an impossibility result for efficiently getting near approximately-stationary points, and the second (in \secref{sec:smoothing}) proving an inherent trade-off between oracle complexity and the smoothness parameter when smoothing nonsmooth functions. The second result also establishes the optimality of randomized smoothing as an efficient smoothing method, under mild assumptions.

Our work leaves open several questions. First, at a more technical level, there is the question of whether some or all of our assumptions in \secref{sec:smoothing} can be relaxed. The result currently requires the algorithm to be translation invariant w.r.t. constant functions, as well as querying at some bounded distance from the input point $\bx$. We conjecture that the translation invariance assumption can be relaxed, possibly by a suitable reduction that shows that any smoothing algorithm can be converted to a translation invariant one. However, how to formally perform this remains unclear at the moment. As to the bounded distance of the queries, it is currently an essential assumption for our proof technique, which relies on a function which looks ``locally'' constant at many different points, but is globally non-constant, and this can generally be determined by querying far enough away from the input point (even along some random direction). Thus, relaxing this assumption may necessitate a different proof technique. 

Another open question is whether randomized smoothing can be ``derandomized'': Our results indicate that the gradient Lipschitz parameter of the smooth approximation cannot be improved, but leave open the possibility of an efficient method returning the actual gradients of some smooth approximation (up to machine precision), in contrast to randomized smoothing which only provides noisy stochastic estimates of the gradient. These can then be plugged into smooth optimization methods which assume access to the exact gradients (rather than noisy stochastic estimates), generally improving the resulting iteration complexity. We note that naively computing the exact gradient of $\tilde{f}(\cdot)$ arising from randomized smoothing is infeasible in general, as it involves a high-dimensional integral.

At a more general level, our work leaves open the question of what is the ``right'' metric to study for nonsmooth-nonconvex optimization, where neither minimizing optimization error nor finding approximately-stationary points is feasible. In this paper, we show that the goal of getting near approximately stationary points is not feasible, at least in the worst case, whereas smoothing can be done efficiently, but not in a dimension-free manner. Can we find other compelling goals to consider? One very appealing notion is the $(\delta,\epsilon)$-stationarity of \citet{zhang2020complexity} that we mentioned in the introduction, which comes with clean, finite-time and dimension-free guarantees. Our negative result in \thmref{thm:main} provides further motivation to consider it, by showing that a natural variation of this notion will not work. However, as we discuss in Appendix \ref{app:de-stationarity}, we need to accept that this stationarity notion can have unexpected behavior, and there exist cases where it will not resemble a stationary point in any intuitive sense. In any case, using an oracle complexity framework to study this and other potential metrics for nonsmooth nonconvex optimization, which combine computational efficiency and finite-time guarantees, remains an interesting direction for future research.

\subsection*{Acknowledgements}
This research is supported by the European Research Council (ERC) grant 754705.

\bibliographystyle{plainnat}
\bibliography{bib}

\appendix

\section{$(\delta,\epsilon)$-Stationarity \citep{zhang2020complexity}}\label{app:de-stationarity}

In the recent work by Zhang, Lin, Jegelka, Sra and Jadbabaie \citep{zhang2020complexity}, the authors prove that for nonconvex nonsmooth functions, finding $\epsilon$-approximately stationary points is infeasible in general. Instead, they study the following relaxation (based on the notion of $\delta$-differential introduced by \citealt{goldstein1977optimization}): Letting $\partial f(\bx)$ denote the generalized gradient set (as defined in \secref{sec:preliminaries}) of $f(\cdot)$ at $\bx$, we say that a point $\bx$ is a \emph{$(\delta,\epsilon)$-stationary point}, if
\begin{equation}\label{eq:destat}
\min\{\norm{\bu}:\bu\in \text{conv}\{\cup_{\by:\norm{\by-\bx}\leq \delta}~\partial f(\by)\}\}~\leq~\epsilon~,
\end{equation}
where $\text{conv}\{\cdot\}$ is the convex hull. In words, there exists a convex combination of gradients at a $\delta$-neighborhood of $\bx$, whose norm is at most $\epsilon$. Remarkably, the authors then proceed to provide a dimension-free, gradient-based algorithm for finding $(\delta,\epsilon)$-stationary points, using $\Ocal(1/\delta\epsilon^3)$ gradient and value evaluations, as well as study related settings. Subsequently, several works have continued the study of optimization in terms of $(\delta,\epsilon)$-stationary points \citep{davis2021gradient,tian22a}.

Although this constitutes a very useful algorithmic contribution to nonsmooth optimization, it is important to note that a $(\delta,\epsilon)$-stationary point $\bx$ (as defined above) \emph{does not} imply that $\bx$ is $\delta$-close to an $\epsilon$-stationary point of $f(\cdot)$, nor that $\bx$ necessarily resembles a stationary point. Intuitively, this is because the convex hull of the gradients might contain a small vector, without any of the gradients being particular small. This is formally demonstrated in the following proposition:

\begin{proposition}\label{prop:stat}
	For any $\delta>0$, there exists a differentiable function $f(\cdot)$ on $\reals^2$ which is $2\pi$-Lipschitz on a ball of radius $2\delta$ around the origin, and the origin is a $(\delta,0)$-stationary point, yet $\min_{\bx:\norm{\bx}\leq \delta}\norm{\nabla f(\bx)}\geq 1$.
\end{proposition}
\begin{figure}\label{fig:stat}
	\centering
	\includegraphics[trim=0cm 1.5cm 0cm 1.5cm,clip=true,width=0.7\linewidth]{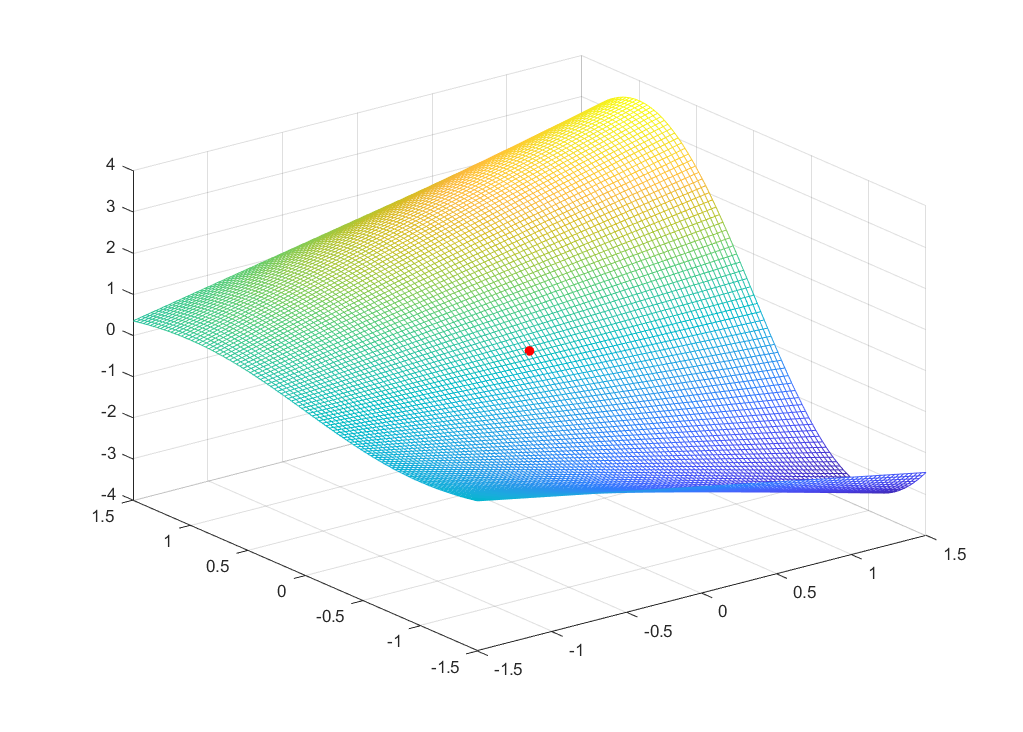}
	\caption{The function used in the proof of Proposition \ref{prop:stat}, for $\delta=1$. The origin (which fulfills the definition of a $(1,0)$-stationary point) is marked with a red dot. Best viewed in color.}
\end{figure}
\begin{proof}
	Fixing some $\delta>0$, consider the function
	\[
	f(u,v)~:=~(2\delta+u)\sin\left(\frac{\pi}{2\delta}v\right)
	\]
	(see \figref{fig:stat} for an illustration). This function is differentiable, and its gradient satisfies
	\[
	\nabla f(u,v)~=~ \left(\sin\left(\frac{\pi}{2\delta}v\right)~,~\frac{\pi}{2\delta}(2\delta+u)\cos\left(\frac{\pi}{2\delta}v\right)\right)~.
	\]
	First, we note that
	\[
	\frac{1}{2}\left(\nabla f(0,\delta)+\frac{1}{2}\nabla f(0,-\delta)\right)~=~
	\frac{1}{2}\left((1,0)+(-1,0)\right)~=~(0,0),
	\]
	which implies that $(0,0)$ is in the convex hull of the gradients at a distance at most $\delta$ from the origin, hence the origin is a $(\delta,0)$-stationary point. Second, we have that
	\begin{equation}\label{eq:nablaf2}
	\norm{\nabla f(u,v)}^2 = \sin^2\left(\frac{\pi}{2\delta}v\right)+\left(\frac{\pi}{2\delta}\right)^2(2\delta+u)^2\cos^2\left(\frac{\pi}{2\delta}v\right)~.
	\end{equation}
	For any $(u,v)$ of norm at most $2\delta$, we must have $|u|\leq 2\delta$, and therefore the above is at most
	\begin{align*}
	\sin^2\left(\frac{\pi}{2\delta}v\right)+\left(\frac{\pi}{2\delta}\right)^2(2\delta+2\delta)^2\cos^2\left(\frac{\pi}{2\delta}v\right)~\leq~ 4\pi^2\left(\sin^2\left(\frac{\pi}{2\delta}v\right)+\cos^2\left(\frac{\pi}{2\delta}v\right)\right)
	~=~ 4\pi^2~,
	\end{align*}
	which implies that the function is $2\pi$-Lipschitz on a ball of radius $2\delta$ around the origin. Finally, for any $(u,v)$ of norm at most $\delta$, we have $|u|\leq \delta$, so \eqref{eq:nablaf2} is at least
	\[
	\sin^2\left(\frac{\pi}{2\delta}v\right)+\left(\frac{\pi}{2\delta}\right)^2(2\delta-\delta)^2\cos^2\left(\frac{\pi}{2\delta}v\right)~\geq~ \sin^2\left(\frac{\pi}{2\delta}v\right)+\cos^2\left(\frac{\pi}{2\delta}v\right)~=~1~.
	\]
\end{proof}

\begin{remark}[Extension to globally Lipschitz functions]
Although the function $f(\cdot)$ in the proof has a constant Lipschitz parameter only close to the origin, it can be easily modified to be globally Lipschitz and bounded, for example by considering the function
\[
\tilde{f}(\bx) ~=~ \begin{cases} f(\bx)& \norm{\bx}\leq 2\delta\\ \max\left\{0,2-\frac{\norm{\bx}}{2\delta}\right\}\cdot f\left(\frac{2\delta}{\norm{\bx}}\bx\right) & \norm{\bx}>2\delta\end{cases}~,
\]
which is identical to $f(\cdot)$ in a ball of radius $2\delta$ around the origin, but decays to $0$ for larger $\bx$, and can be verified to be globally bounded and Lipschitz independent of $\delta$. 
\end{remark}

\begin{remark}[Extension to constant distances]\label{remark:propstrong}
    The proof of \thmref{thm:main} uses a (more complicated) construction that actually strengthens Proposition \ref{prop:stat}: It implies that for any $\delta,\epsilon$ smaller than some constants, there is a Lipschitz, bounded-from-below function on $\reals^d$, such that the origin is $(\delta,0)$-stationary, yet there are no $\epsilon$-stationary points even at a constant distance from the origin. In more details, consider the function 
	\begin{gather*}
	\hat{g}_{\bw}(\bx):=\max\{g_{\bw}(\mathbf{0})-1~,~g_{\bw}(\bx)\}~,
	\\
	g(\bx)_{\bw}=|x_d|+\frac{1}{4}\sqrt{\sum_{i=1}^{d-1}x_i^{2}}-\left[\langle\overline{\bw},\bx+\bw\rangle-\frac{1}{2}\|\bx+\bw\|\right]_{+}~.
	\end{gather*}
	Using exactly the same proof as for \lemref{lem: f_w properties}, one can show that
	$g_{\bw}(\cdot)$ is $\frac{15}{4}$-Lipschitz and has no $\epsilon$-stationary points for $\epsilon < 1/4\sqrt{2}$. Therefore, it is easily verified that for any $\bw$, $\hat{g}_{\bw}(\cdot)$ is $\frac{15}{4}$-Lipschitz, bounded from below, and any $\epsilon$-stationary point is at a distance of at least $4/15$ from the origin.\footnote{The last point follows from the fact that if $\by$ is an $\epsilon$-stationary point of $\hat{g}_{\bw}(\cdot)$, then we can find a point $\bx$ arbitrarily close to $\by$ such that $\hat{g}_{\bw}(\bx)\neq g_{\bw}(\bx)$, hence $g_{\bw}(\bx)< g_{\bw}(\mathbf{0})-1$, and as a result $g_{\bw}(\mathbf{0})-f_{\bw}(\bx)>1$. But $g_{\bw}(\cdot)$ is $\frac{15}{4}$-Lipschitz, hence $\norm{\bx}> 4/15$, and therefore $\norm{\by}\geq 4/15$.} However, we also claim that the origin is a $(\delta,0)$-stationary point for any $\delta\in \left(0,4/15\right)$. To see this, note first that for such $\delta$, by the Lipschitz property of $g_{\bw}(\cdot)$, we have $\hat{g}_{\bw}(\bx)=g_{\bw}(\bx)$ in a $\delta$-neighborhood of the origin. Fix any $\bw$ such that $\norm{\bw}=\frac{\delta}{2},\,w_d=0$, and let $\bv\in\{-\delta\cdot\mathbf{e}_d,\delta\cdot\mathbf{e}_d\}$. It is easily verified that $\inner{\bar{\bw},\bv+\bw}-\frac{1}{2}\norm{\bv+\bw}<0$, in which case
\[
\sign{(v_d)}\cdot\mathbf{e}_d \in
\partial g_{\bw}(\bv)~=~
\partial \hat{g}_{\bw}(\bv)~,
\]
and therefore $\frac{1}{2}\left(\nabla \hat{g}_{\bw}(\bv)+\nabla \hat{g}_{\bw}(-\bv)\right)=\mathbf{0}$, where $\nabla \hat{g}_{\bw}(\bv)$ denotes the subgradient defined above.
\end{remark}

We end by noting that if we drop the the $\text{conv}\{\cdot\}$ operator from the definition of $(\delta,\epsilon)$-stationarity in \eqref{eq:destat}, the goal becomes equivalent to finding points which are $\delta$-close to $\epsilon$-approximately stationary points -- which is exactly the goal we study in \secref{sec:near-approximate}, and for which we show a strong impossibility result. This impossibility result implies that a natural strengthening of the notion of $(\delta,\epsilon)$-stationarity is  already too strong to be feasible in general.

\section{Runtime of smoothed GD suffers from a dimension dependency} \label{app: smoothed GD dimension}

In this appendix, we formally prove that randomized smoothing can indeed lead to strong dimension dependencies in the iteration complexity of simple gradient methods -- in particular, vanilla gradient descent with constant step size --  even for simple convex functions. Thus, the dimension dependency arising from applying gradient descent on a randomly-smoothed function is real and not merely an artifact of the analysis (where the standard upper bound on the number of iterations scales with the gradient Lipschitz parameter). We note that we focus on constant step-size gradient descent for simplicity, and a similar analysis can be performed for other gradient-based methods, such as variable step-size gradient descent or stochastic gradient descent. 

Given a 1-Lipschitz function $f:\reals^d\to\reals$,
denote the smooth approximation $\tilde{f}(\bx)=\E_{\|\bv\|\leq1}[f(\bx+\epsilon\bv)]$ where $\bv$ is distributed uniformly over the unit ball. Let $\bx_0$ be a point which is of distance at most 1 to an $\epsilon$-stationary point of $\tilde{f}$, and consider vanilla gradient descent with a constant step size $\eta>0$:
\[
\bx_{t+1}=\bx_{t}-\eta\cdot \nabla\tilde{f}\left(\bx_{t}\right)~.
\]
The following proposition shows that for any step size, applying gradient descent to find an approximately-stationary point of $\tilde{f}$ will necessitate a number of iterations scaling strongly with the dimension:
\begin{proposition} \label{prop:smoothed GD dimension}
    There exists a 1-Lipschitz function $f:\reals^d\to\reals$ such that the following holds: For any $\epsilon<\frac{1}{2},~\eta>0$, there exists $\bx_0$ as above such that
    $\min\{t:\|\nabla\tilde{f}(\bx_t)\|\leq\epsilon\}=\Omega\left(\frac{\sqrt{d}}{\epsilon}\right)$.
\end{proposition}
\begin{proof}
    We will show the claim holds for $f\left(\bx\right):=\left|x_1\right|$. 
    In a nutshell, the proof is based on the observation that $\nabla \tilde{f}(\bx)$ is close to zero only when $|x_1|=\Ocal(1/\sqrt{d})$. Thus, gradient descent must hit an interval of size $\Ocal(1/\sqrt{d})$. But in order to guarantee this, and with an arbitrary bounded starting point, the step size must be small, and hence the number of iterations required will be large. 
    
    Proceeding with the formal proof, note that $\tilde{f}\left(\bx\right)=\E_{\|\bv\|\leq1}\left[\left|x_1+\epsilon v_1\right|\right]$, hence
    \begin{align}
        \nabla\tilde{f}\left(\bx\right)
        &=\E_{\|\bv\|\leq1}\left[\sign\left(x_1+\epsilon v_1\right)\right]\cdot\mathbf{e}_1 \nonumber\\
        &=\left(\Pr_{\|\bv\|\leq1}\left[x_1+\epsilon v_1>0\right]-\Pr_{\|\bv\|\leq1}\left[x_1+\epsilon v_1<0\right]\right)\cdot\mathbf{e}_1 \nonumber\\
        &=\left(1-2\cdot\Pr_{\|\bv\|\leq1}\left[x_1+\epsilon v_1<0\right]\right)\cdot\mathbf{e}_1 \nonumber\\
        &=\left(1-2\cdot\Pr_{\|\bv\|\leq1}\left[v_1<-\frac{x_1}{\epsilon}\right]\right)\cdot\mathbf{e}_1
        ~.
        \label{eq: smoothed grad}
    \end{align}
    We draw several consequences from \eqref{eq: smoothed grad}. First, if $x_1=0$ then $\Pr_{\|\bv\|\leq1}\left[v_1<-\frac{x_1}{\epsilon}\right]=\frac{1}{2}$ due to symmetry around the origin, so in particular
    \begin{equation} \label{eq: 0 global min}
        \nabla\tilde{f}\left(\mathbf{0}\right)=\mathbf{0}~.
    \end{equation}
    Second, if $x_1\geq\epsilon$ then $\Pr_{\|\bv\|\leq1}\left[v_1<-\frac{x_1}{\epsilon}\right]=0$, and if $x_1\leq-\epsilon$ then $\Pr_{\|\bv\|\leq1}\left[v_1<-\frac{x_1}{\epsilon}\right]=1$. Overall
    \begin{equation} \label{eq: grad=1 far from zero}
        \left|x_1\right|\geq\epsilon\implies\nabla\tilde{f}\left(\bx\right)=\sign(x_1)\cdot\mathbf{e}_1~.
    \end{equation}
    Third, since probabilities are bounded between zero and one, we obtain the global upper estimate
    \begin{equation} \label{eq: grad norm<=1}
        \left\|\nabla\tilde{f}\left(\bx\right)\right\|\leq1~.
    \end{equation}
    Lastly, $\Pr_{\|\bv\|\leq1}\left[v_1<-\frac{x_1}{\epsilon}\right]$ equals to the volume of the intersection of the halfspace $\{\bv\in\reals^d~|~ v_1<-\frac{x_1}{\epsilon}\}$ with the unit ball, normalized by the unit ball volume. In particular, since this intersection is a subset of the spherical sector associated with the spherical cap $\left\{\bv\in\SSS^{d-1}~\middle|~v_1<-\frac{x_1}{\epsilon}\right\}$, its normalized volume is less then the surface area of the cap. By well known estimates of spherical cap (for example \citealp[Lemma 2.2]{ball1997elementary}):
    \begin{equation} \label{eq: spherical cap}
    \Pr_{\|\bv\|\leq1}\left[v_1<-\frac{x_1}{\epsilon}\right]
    \leq\Pr_{\|\bv\|=1}\left[v_1<-\frac{x_1}{\epsilon}\right]\leq\exp\left(-\frac{d{x}_1^2}{2\epsilon^2}\right)~.
    \end{equation}
    By combining \eqref{eq: smoothed grad} and \eqref{eq: spherical cap} we get
    \[
    \left\|\nabla\tilde{f}\left(\bx\right)\right\|\geq
    1-2\exp\left(-\frac{d{x}_1^2}{2\epsilon^2}\right)~.
    \]
    In particular,
        \begin{equation} \label{eq: large norm}
        \left|x_1\right|\geq\frac{\sqrt{2\log(10)}\epsilon}{\sqrt{d}}
        \implies
        \left\|\nabla\tilde{f}\left(\bx\right)\right\|\geq \frac{4}{5}~.
    \end{equation}
    We are now ready to describe the choice of $\bx_0$ which will prove the claim, depending on the value of $\eta$.
    \subsubsection*{Case I: $\eta\leq\frac{5\sqrt{2\log(10)}\epsilon}{2\sqrt{d}}$}
    We set $\bx_0=\mathbf{e}_1$. First, $\bx_0$ is indeed at distance 1 from $\mathbf{0}$, which by \eqref{eq: 0 global min} is a stationary point. Furthermore, by the definition of gradient descent, \eqref{eq: smoothed grad} and \eqref{eq: grad norm<=1}, for all $t\leq\frac{2\sqrt{d}}{5\sqrt{2\log(10)}\epsilon}-\frac{2}{5}$:
    \begin{align*}
    \left(\bx_{t+1}\right)_1&=\left(\bx_0-\eta\left(\sum_{i=1}^{t}\nabla\tilde{f}\left(\bx_i\right)\right)\right)_1 \\
    &\geq1-\frac{5\sqrt{2\log(10)}\epsilon}{2\sqrt{d}}\cdot t\cdot1
    \\
    &\geq\frac{\sqrt{2\log(10)}\epsilon}{\sqrt{d}}~.
    \end{align*}
    So by \eqref{eq: large norm}, for every $t\leq\frac{2\sqrt{d}}{5\sqrt{2\log(10)}\epsilon}-\frac{2}{5}:\ \left\|\nabla\tilde{f}\left(\bx_{t}\right)\right\|\geq\frac{4}{5}$. Consequently, the minimal $t$ for which the gradient norm is less than $\epsilon$ satisfies $t>\frac{2\sqrt{d}}{5\sqrt{2\log(10)}\epsilon}-\frac{2}{5}=\Omega(\frac{\sqrt{d}}{\epsilon})$.
    \subsubsection*{Case II: $\frac{5\sqrt{2\log(10)}\epsilon}{2\sqrt{d}}<\eta\leq2$}
    In this case, we define the real function
    \[
    \phi\left(s\right):=2s-\eta\left(\nabla\tilde{f}\left(s\cdot\mathbf{e}_1\right)\right)_1
    ~.
    \]
    On on hand, by assumption on $\eta$ and \eqref{eq: large norm}:
    \begin{align*}
    \phi\left(\frac{\sqrt{2\log(10)}\epsilon}{\sqrt{d}}\right)
    &=\frac{2\sqrt{2\log(10)}\epsilon}{\sqrt{d}}-
    \eta\left(\nabla\tilde{f}\left(s\cdot\mathbf{e}_1\right)\right)_1 \\
    &\leq\frac{2\sqrt{2\log(10)}\epsilon}{\sqrt{d}}-\frac{5\sqrt{2\log(10)}\epsilon}{2\sqrt{d}}\cdot\frac{4}{5}\\
    &=0
    ~.
    \end{align*}
    On the other hand, $\frac{\eta}{2}>\frac{5\sqrt{2\log(10)}\epsilon}{4\sqrt{d}}>\frac{\sqrt{2\log(10)}\epsilon}{\sqrt{d}}$ and by \eqref{eq: grad norm<=1}:
    \begin{align*}
    \phi\left(\frac{\eta}{2}\right)
    &=\eta-\eta\left(\nabla\tilde{f}\left(\frac{\eta}{2}\cdot\mathbf{e}_1\right)\right)_1 \\
    &\geq\eta-\eta\cdot1 \\
    &=0
    ~.
    \end{align*}
    Notice that $\phi$ is continuous since $\tilde{f}$ is smooth, so by the intermediate value theorem there exists $s^*\in\left[\frac{\sqrt{2\log(10)}\epsilon}{\sqrt{d}},\frac{\eta}{2}\right]$ such that $\phi\left(s^*\right)=0$. Equivalently,
    \begin{equation} \label{eq: s^*}
        s^*-\eta\left(\nabla\tilde{f}\left(s\cdot\mathbf{e}_1\right)\right)_1=-s^*~.
    \end{equation}
    We set $\bx_0=s^*\mathbf{e}_1$. First, $\bx_0$ is of distance at most $\frac{\eta}{2}\leq1$ from $\mathbf{0}$, which by \eqref{eq: 0 global min} is a stationary point. Furthermore, by the definition of gradient descent and \eqref{eq: s^*} we get
    \[
    \bx_1=s^*\mathbf{e}_1-\eta\nabla\tilde{f}\left(s^*\mathbf{e}_1\right)=-s^*\mathbf{e}_1=-\bx_0~.
    \]
    Inductively, due to the symmetry of $\tilde{f}$ with respect to the origin, we obtain $\bx_t=(-1)^{t}\bx_0$. In particular, since $s^*\geq\frac{\sqrt{2\log(10)}\epsilon}{\sqrt{d}}$ \eqref{eq: large norm} ensures that for all $t\in\mathbb{N}:\ \left\|\nabla\tilde{f}\left(\bx_t\right)\right\|\geq\frac{4}{5}>\epsilon$.
    \subsubsection*{Case III: $\eta>2$}
    Set $\bx_0=\mathbf{e}_1$, which satisfies the distance assumption as explained in case I. By the definition of gradient descent and \eqref{eq: grad=1 far from zero}:
    \[
    \bx_1=\mathbf{e}_1-\eta\nabla\tilde{f}\left(\mathbf{e}_1\right)
    =\left(1-\eta\right)\mathbf{e}_1~.
    \]
    Notice that $\left(1-\eta\right)<-1$, so by invoking \eqref{eq: grad=1 far from zero} we get
    \[
    \bx_2=\bx_1-\eta\nabla\tilde{f}\left(\bx_1\right)
    =\left(1-\eta\right)\mathbf{e}_1+\eta\mathbf{e}_1=\bx_0~.
    \]
    We deduce that for all $t\in\mathbb{N}:\ \bx_{t+2}=\bx_t$, and in particular by \eqref{eq: grad=1 far from zero}: $\left\|\nabla\tilde{f}\left(\bx_t\right)\right\|=1>\epsilon$.
\end{proof}

\section{Proof of \propref{prop: hard 1dim}} \label{app: 1dim hardness proof}
Denote by $\Hcal$ the set of non-negative $2$-Lipschitz functions $h$ such that $h(0)=1$, $x^*:=\arg\min_{x\in\reals}h(x)\in(0,1)$ is unique, and $\forall x\neq x^*~\forall g\in\partial h(x):|g|\geq1$.
We start by showing that if the proposition does not hold, then there exists an algorithm that finds the minimum of any function in $\Hcal$ within some finite time, with high probability. We will use this implication in order to obtain a contradiction.

Assume by contradiction that \propref{prop: hard 1dim} does not hold. That is, that there exist $\Acal,T,\delta$ such that for any $h\in\Hcal,\rho>0$,
\begin{equation} \label{eq: contradiction assumption}
   \Pr_{\Acal}\left[\min_{t\in[T]}|x_t^{h}-x^*| < \rho\right]\geq\delta~. 
\end{equation}
Let $(\rho_n)_{n=1}^{\infty}>0$ be a decreasing sequence such that $\lim_{n\to\infty}\rho_{n}=0$. By continuity of probability measures with respect to decreasing events, assuming \eqref{eq: contradiction assumption} for any $\rho>0$ implies
\begin{align} \label{eq: exist T,delta}
   \Pr_{\Acal}\left[\exists t\in[T]:x_t^{h}=x^*\right]
   &=\Pr_{\Acal}\left[\min_{t\in[T]}|x_t^{h}-x^*| =0\right]
   =\Pr_{\Acal}\left[\bigcap_{n=1}^{\infty}\min_{t\in[T]}|x_t^{h}-x^*| <\rho_n\right] \nonumber\\
   &=\lim_{n\to\infty}\Pr_{\Acal}\left[\min_{t\in[T]}|x_t^{h}-x^*| <\rho_n\right]
   \geq\delta~. 
\end{align}
Namely, there exists some algorithm $\Acal$ that gets to the exact minimum of any function in $\Hcal$ within some finite time $T$, with some positive probability $\delta$. But note that a classic confidence boosting argument shows that if \eqref{eq: exist T,delta} holds for \emph{some} $0<\delta<1$, then it actually holds for \emph{any} $0<\delta<1$ (with an appropriate blow up in running time). Indeed, assuming it is true for some $\delta_{0},\Acal_0,T_0$, we define an algorithm $\Acal$ which simulates $N$ independent copies of $\Acal_{0}$, and returns the point with the smallest function value over all seen iterates along all the independent copies. By standard Chernoff-Hoeffding bounds one easily gets that for $N$ being some function of $\delta_0$ and $\delta$, $\Acal$ satisfies \eqref{eq: exist T,delta} for $T=N\cdot T_{0}$. Hence,
\begin{equation} \label{eq: any delta}
    \forall\delta<1
    ~\exists\Acal_{\delta},T_{\delta}
    ~\forall h\in\Hcal~:~
    \Pr_{\Acal_{\delta}}\left[\exists t\in[T]:x_t^{h}=x^*\right]\geq\delta~.
\end{equation}

We fix $\delta=\frac{1}{2}$, and let $\Acal,T_{0}$ be it's associated algorithm and iteration number.
By Yao's lemma \citep{yao1977probabilistic}, we can assume $\Acal$ is deterministic and provide a distribution over hard functions.
Namely, in order to arrive at a contradiction it is enough to show that
\begin{equation} \label{eq: Yao hard distribution}
\Pr_{\sigma}\left[\exists t\in[T_{0}]:x_{t}^{h_\sigma}=x^*\right]<\frac{1}{2}~,
\end{equation}
for some distribution over $\sigma$, such that $\forall\sigma:h_\sigma\in\Hcal$.

Before delving into the technical details, we turn to explain the intuition behind the construction. We consider functions $h^{N}_{\sigma}$ indexed by $\sigma\in\{0,1\}^N,N\in\NN$. The function $h_{\sigma_1}^{1}$ ``tilts to the left'' if $\sigma_1=0$, and ``tilts to the right'' if $\sigma_1=1$ (see \figref{fig:h_1dim_hard}). Given these two functions, we define the functions for $N=2$, such that $\sigma_1$ determines an ``outer'' tilt, while $\sigma_2$ determines an ``inner'' tilt which behaves like $h^1_{\sigma_2}$ (once again, see \figref{fig:h_1dim_hard}). These functions are such that for any point outside the outer tilted segment, it's value does not depend on the inner tilt - that is, on $\sigma_2$. We continue this process recursively for any $N\in\NN$, such that the larger $i$ is, $\sigma_i$ determines finer tilts in smaller segments. The construction has the property that for all points outside the tilt determined by $\sigma_i$, the function's values do not depend on $\sigma_{i+1},\dots,\sigma_{N}$. Thus, by setting $N$ large enough relatively to $T$, we will be able to ensure that $x_1,\dots,x_T$ are not likely to depend on $\sigma_{N}$, therefore missing the minimum which does depend on it. To that end, we define the following functions:

\begin{center}
\begin{tabular}{ c c c }
$h_{0}^{1}\left(x\right)=
    \begin{cases}
    1-x & x\in\left(-\infty,0\right)\\
    1-2x & x\in\left[0,\frac{3}{8}\right]\\
    \frac{6}{5}x-\frac{1}{5} & x\in\left[\frac{3}{8},1\right]\\
    x & x\in\left(1,\infty\right)
    \end{cases}$
 &
 ,
 &
 $h_{1}^{1}\left(x\right)=
    \begin{cases}
    1-x & x\in\left(-\infty,0\right)\\
    -\frac{6}{5}x+1 & x\in\left[0,\frac{5}{8}\right]\\
    2x-1 & x\in\left[\frac{5}{8},1\right]\\
    x & x\in\left(1,\infty\right)
    \end{cases}$
    ~~.
\end{tabular}
\end{center}
Next, in a recursive manner, for any $\hat{\sigma}:=(\sigma_2,\dots,\sigma_{N})\in\{0,1\}^{N-1}$ we define (see \figref{fig:h_1dim_hard}):

\begin{align*}
h_{0,\hat{\sigma}}^{N}\left(x\right)&=\begin{cases}
    1-x & x\in\left(-\infty,0\right)\\
    1-2x & x\in\left[0,\frac{1}{4}\right]\\
    \frac{1}{4}h_{\hat{\sigma}}^{(N-1)}\left(4x-1\right)+\frac{1}{4} & x\in\left[\frac{1}{4},\frac{1}{2}\right]\\
    x & x\in\left[\frac{1}{2},1\right]\\
    x & x\in\left(1,\infty\right)
    \end{cases}
    ~,
    \\
h_{1,\hat{\sigma}}^{N}\left(x\right)&=
    \begin{cases}
    1-x & x\in\left(-\infty,0\right)\\
    1-x & x\in\left[0,\frac{1}{2}\right]\\
    \frac{1}{4}h_{\hat{\sigma}}^{(N-1)}\left(4x-2\right)+\frac{1}{4} & x\in\left[\frac{1}{2},\frac{3}{4}\right]\\
    2x-1 & x\in\left[\frac{3}{4},1\right]\\
    x & x\in\left(1,\infty\right)
    \end{cases}
    ~.
\end{align*}

\begin{figure}
	\includegraphics[trim=3cm 8cm 1cm 8.5cm,clip=true,width=0.5\linewidth]{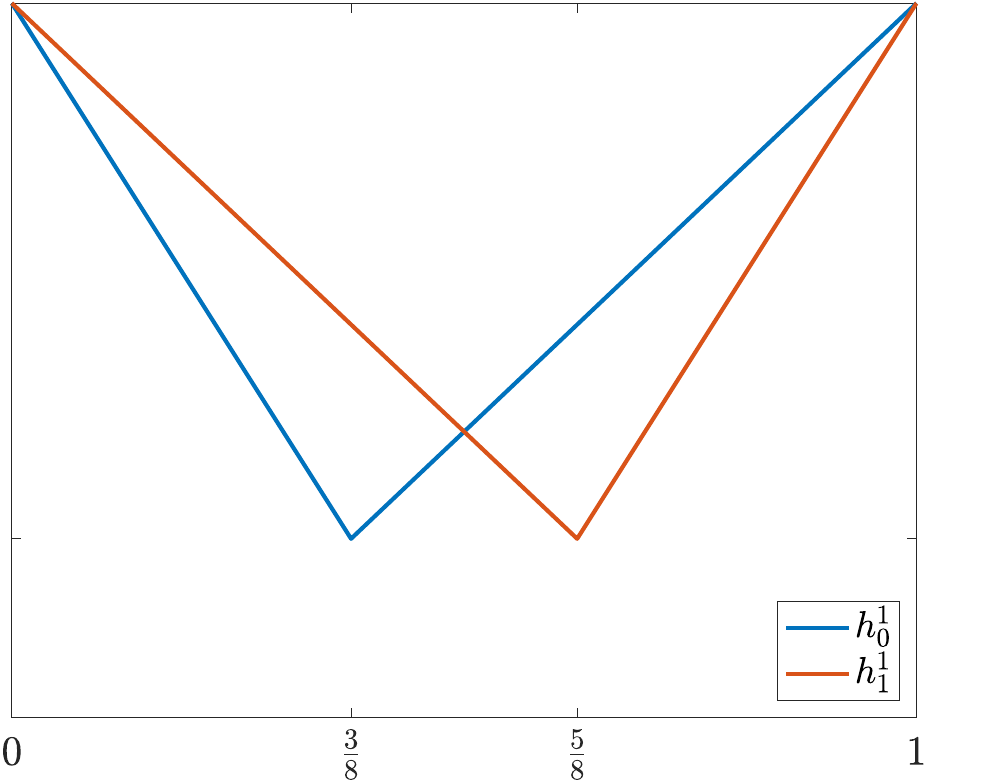}%
	\includegraphics[trim=3cm 8cm 1cm 8.5cm,clip=true,width=0.5\linewidth]{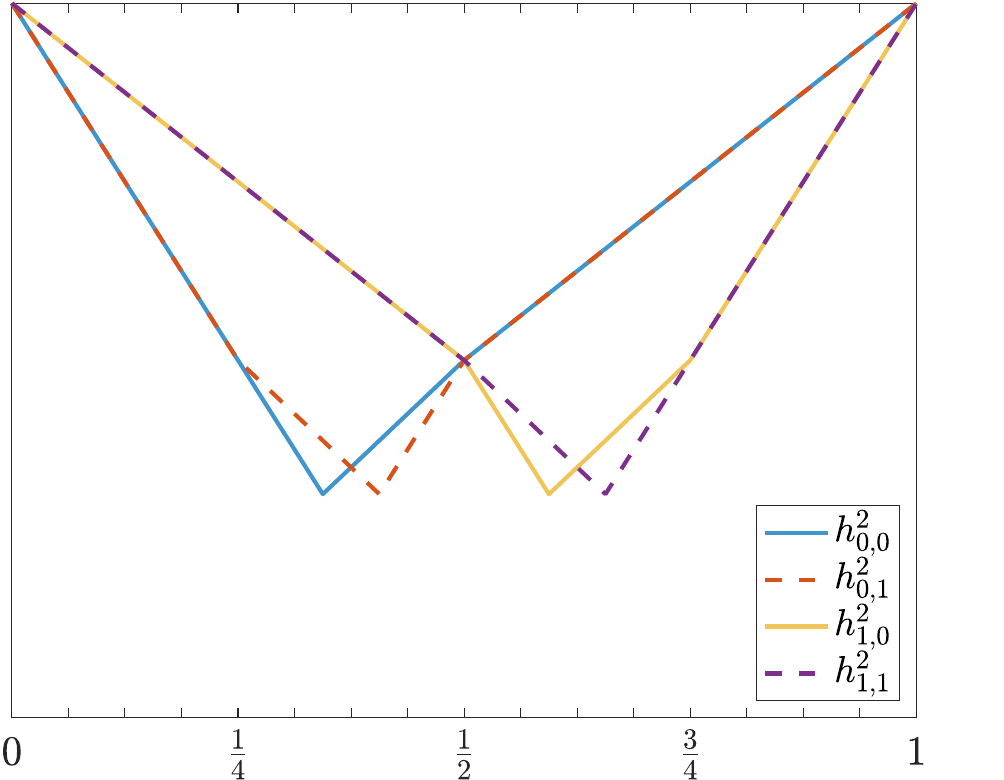}
	\caption{Plots of the functions $h^{N}_{\sigma},\sigma\in\{0,1\}^N$, $N=1$ (on the left) and $N=2$ (on the right).}
	\label{fig:h_1dim_hard}
\end{figure}

\begin{lemma}
For any $N\in\NN$, it holds that for all $\sigma\in\{0,1\}^{N}:~h^{N}_{\sigma}\in\Hcal$.
\end{lemma}
\begin{proof}
The proof is by induction on $N$. For $N=1$ it is easy to verify that $h^1_0,h^1_1\in\Hcal$, and $h^1_0(1)=h^1_1(1)=1$, as well as the fact that $h^1_0,h^1_1$ are both piecewise linear. Assume the claim is true for some $N-1$, and that for all $\hat{\sigma}\in\{0,1\}^{N-1}:h^{N-1}_{\hat{\sigma}}(1)=1$ as well as that they are all piecewise linear. Consider $h^{N}_{0,\hat{\sigma}}$ for some $\hat{\sigma}\in\{0,1\}^{N-1}$. First, $h^{N}_{0,\hat{\sigma}}(0)=1-2\cdot0=1$ as required. Moreover, it is clear by definition that $h^{N}_{0,\hat{\sigma}}(x)\geq0$ for all $x\notin[\frac{1}{4},\frac{1}{2}]$. For $x\in[\frac{1}{4},\frac{1}{2}]$, we have by the induction hypothesis
\[
h^{N}_{0,\hat{\sigma}}(x)
=\frac{1}{4}\cdot\underset{\geq0}{\underbrace{h^{N-1}_{\hat{\sigma}}(4x-1)}}+\frac{1}{4}\geq0~,
\]
which establishes the required non-negativity property. Note that $h^{N}_{0,\hat{\sigma}}$ is continuous. Indeed, it is easy to verify continuity for any $x\notin\{\frac{1}{4},\frac{1}{2}\}$, and for those two points we have
\begin{align*}
    h^{N}_{0,\hat{\sigma}}\left(\frac{1}{4}\right)
    &=\frac{1}{4}h^{N-1}_{\hat{\sigma}}\left(0\right)+\frac{1}{4}
    =\frac{1}{4}\cdot1+\frac{1}{4}
    =\frac{1}{2}
    =\lim_{x\to\frac{1}{4}}(1-2x)
    \\
    h^{N}_{0,\hat{\sigma}}\left(\frac{1}{2}\right)
    &=\frac{1}{4}h^{N-1}_{\hat{\sigma}}\left(1\right)+\frac{1}{4}
    =\frac{1}{4}\cdot1+\frac{1}{4}
    =\frac{1}{2}
    =\lim_{x\to\frac{1}{2}}(x)~.
\end{align*}
Using our induction hypothesis we see that $h^{N}_{0,\hat{\sigma}}$ is a piecewise linear continuous function. Thus, in order to prove the remaining properties 
it is enough to inspect $\partial h^{N}_{0,\hat{\sigma}}$. It is easy to verify that for all $x\notin[\frac{1}{4},\frac{1}{2}]:~\partial h^{N}_{0,\hat{\sigma}}(x)\subseteq\{-1,-2,1\}$. For $x\in[\frac{1}{4},\frac{1}{2}]$, we have
\begin{align}
\partial h^{N}_{0,\hat{\sigma}}(x)
&=\partial\left(\left(z\mapsto\frac{1}{4}z+\frac{1}{4}\right)\circ\left(y\mapsto h^{N-1}_{\hat{\sigma}}(y)\right)\circ\left(x\mapsto4x-1\right)\right)(x)
\nonumber\\
&=\left\{\frac{1}{4}\cdot g\cdot4\middle|~g\in\partial h^{N-1}_{\hat{\sigma}}(4x-1)\right\}
\nonumber\\
&=\left\{g\middle|~g\in\partial h^{N-1}_{\hat{\sigma}}(4x-1)\right\}~. \label{eq: derivative set}
\end{align}
By the induction hypothesis, all elements of the set above are of magnitude at most $2$, which establishes the desired Lipschitz property. Furthermore, note that if $x\in[\frac{1}{4},\frac{1}{2}]$ then $4x-1\in[0,1]$. Using the induction hypothesis, let $x^*\in[\frac{1}{4},\frac{1}{2}]$ be the unique number such that $4x^*-1=\arg\min h^{N-1}_{\hat{\sigma}}$. Then by the induction hypothesis and \eqref{eq: derivative set}, for all $x\in[\frac{1}{4},\frac{1}{2}]\setminus\{x^*\}~\forall g\in\partial h^{N}_{0,\hat{\sigma}}(x):|g|\geq1$ which is exactly the desired property, assuming we show that $x^*=\arg\min h^{N}_{0,\hat{\sigma}}$. Finally, noticing that $h^{N}_{0,\hat{\sigma}}$ is decreasing for all $x<x^*$ and increasing for all $x>x^*$ (which inside the interval $[\frac{1}{4},\frac{1}{2}]$ follows once again from the induction hypothesis and our choice of $x^*$) finishes the proof for $h^{N}_{0,\hat{\sigma}}$. The proof for $h^{N}_{1,\hat{\sigma}}$ is essentially the same.
\end{proof}
We define
\[
I_{\sigma_1,\dots,\sigma_k}
:=\left[\sum_{i=1}^{k}\frac{\sigma_{i}+1}{4^{i}},~
\sum_{i=1}^{k}\frac{\sigma_{i}+1}{4^{i}}+\frac{1}{4^{k}}\right]
\subset\reals~.
\]
\begin{lemma} \label{lem: I inclusion}
For any $l<k$ and any $\sigma_{1},\dots,\sigma_{l},\dots,\sigma_{k}\in\{0,1\}:
I_{\sigma_1,\dots,\sigma_l}\supset I_{\sigma_1,\dots,\sigma_{l},\dots,\sigma_k}$.
\end{lemma}
\begin{proof}
On one hand,
\begin{equation*}
\sum_{i=1}^{l}\frac{\sigma_{i}+1}{4^{i}}
<\sum_{i=1}^{l}\frac{\sigma_{i}+1}{4^{i}}
+\sum_{i=l+1}^{k}\frac{\sigma_{i}+1}{4^{i}}
=\sum_{i=1}^{k}\frac{\sigma_{i}+1}{4^{i}}~.
\end{equation*}
On the other hand,
\begin{align*}
    \sum_{i=1}^{k}\frac{\sigma_{i}+1}{4^{i}}+\frac{1}{4^k}
    &=\sum_{i=1}^{l}\frac{\sigma_{i}+1}{4^{i}}+\sum_{i=l+1}^{k}\frac{\sigma_{i}+1}{4^{i}}+\frac{1}{4^k}
    \leq\sum_{i=1}^{l}\frac{\sigma_{i}+1}{4^{i}}+\sum_{i=l+1}^{k}\frac{2}{4^{i}}+\frac{1}{4^k}\\
    &=\sum_{i=1}^{l}\frac{\sigma_{i}+1}{4^{i}}
    +\frac{2}{3}\left(\frac{1}{4^l}-\frac{1}{4^k}\right)
    +\frac{1}{4^k}
    =\sum_{i=1}^{l}\frac{\sigma_{i}+1}{4^{i}}
    +\frac{2}{3}\cdot\frac{1}{4^l}+\frac{1}{3}\cdot\frac{1}{4^k}
    \\
    &<\sum_{i=1}^{l}\frac{\sigma_{i}+1}{4^{i}}
    +\frac{2}{3}\cdot\frac{1}{4^l}+\frac{1}{3}\cdot\frac{1}{4^l}
    =\sum_{i=1}^{l}\frac{\sigma_{i}+1}{4^{i}}+\frac{1}{4^l}~.
\end{align*}
\end{proof}

\begin{lemma} \label{lem: I disjoint}
If $(\sigma_1,\dots,\sigma_{k-1})\neq(\sigma'_1,\dots,\sigma'_{k-1})$ then for all $\sigma_k,\sigma'_k\in\{0,1\}:\ I_{\sigma_1,\dots,\sigma_k}\cap I_{\sigma'_1,\dots,\sigma'_k}=\emptyset$. 
\end{lemma}
\begin{proof}
Let $i_0\leq k-1$ the the minimal index $i$ for which $\sigma_{i}\neq\sigma'_{i}$. Assume without loss of generality that $\sigma_{i_0}=0,\sigma'_{i_0}=1$. If $x\in I_{\sigma_1,\dots,\sigma_{k}}$ then by definition
\begin{align}
x&\leq\sum_{i=1}^{k}\frac{\sigma_{i}+1}{4^{i}}+\frac{1}{4^{k}}\nonumber\\
&=\sum_{i=1}^{i_0-1}\frac{\sigma_{i}+1}{4^{i}}
+\frac{\sigma_{i_0}+1}{4^{i_0}}
+\sum_{i=i_0+1}^{k}\frac{\sigma_{i}+1}{4^{i}}
+\frac{1}{4^{k}}\nonumber\\
&\leq\sum_{i=1}^{i_0-1}\frac{\sigma'_{i}+1}{4^{i}}
+\frac{1}{4^{i_0}}
+\sum_{i=i_0+1}^{k}\frac{2}{4^{i}}
+\frac{1}{4^{k}}
\nonumber\\
&=\sum_{i=1}^{i_0-1}\frac{\sigma'_{i}+1}{4^{i}}
+\frac{1}{4^{i_0}}
+\frac{2}{3}\left(\frac{1}{4^{i_0}}-\frac{1}{4^k}\right)
+\frac{1}{4^{k}}
\nonumber\\
&=\sum_{i=1}^{i_0-1}\frac{\sigma'_{i}+1}{4^{i}}
+\frac{2}{4^{i_0}}
+\frac{1}{3}\cdot\frac{1}{4^k}
-\frac{1}{3}\cdot\frac{1}{4^{i_0}}~. \label{eq: interval upper bound}
\end{align}
Using the fact that $i_{0}<k$ we also get
\begin{equation} \label{eq: k,i_0 ineq}
    \frac{1}{3}\cdot\frac{1}{4^k}-\frac{1}{3}\cdot\frac{1}{4^{i_0}}
    <\frac{1}{3}\cdot\frac{1}{4^{i_0}}-\frac{1}{3}\cdot\frac{1}{4^k}
    =\sum_{i=i_0+1}^{k}\frac{1}{4^i}
    \leq\sum_{i=i_0+1}^{k}\frac{\sigma'_{i}+1}{4^i}~.
\end{equation}
Plugging \eqref{eq: k,i_0 ineq} into \eqref{eq: interval upper bound} reveals than for any $x\in I_{\sigma_1,\dots,\sigma_{k}}$:
\begin{equation*}
    x<
    \sum_{i=1}^{i_0-1}\frac{\sigma'_{i}+1}{4^{i}}+\frac{2}{4^{i_0}}
    +\sum_{i=i_0+1}^{k}\frac{\sigma'_{i}+1}{4^i}
    =\sum_{i=1}^{k}\frac{\sigma'_{i}+1}{4^{i}}~.
\end{equation*}
Hence $x\notin I_{\sigma'_1,\dots,\sigma'_k}$, which finishes the proof.
\end{proof}

\begin{lemma} \label{lem: h outside I}
For any $N\geq2$, any
$1\leq k<N$ and any local oracle $\OO$, it holds that $\OO_{h^{N}_{\sigma_1,\dots,\sigma_{N}}}(x)$ does not depend on $\sigma_{k+1},\dots,\sigma_{N}$ for $x\notin I_{\sigma_1,\dots,\sigma_k}$.
\end{lemma}

\begin{proof}
First, notice that since $\reals\setminus I_{\sigma_1,\dots,\sigma_k}$ is an open set it is enough to prove that the function value $h^{N}_{\sigma_1,\dots,\sigma_{N}}(x)$ does not depend on $\sigma_{k+1},\dots,\sigma_{N}$ for $x\notin I_{\sigma_1,\dots,\sigma_k}$, which implies the desired claim about $\OO_{h^{N}_{\sigma_1,\dots,\sigma_{N}}}(x)$ by definition of a local oracle.
We will prove the claim for all natural pairs $(N,k)$ such that $k<N$, using the following inductive argument:
\begin{itemize}
    \item For any $N\geq2$, the claim holds for $(N,1)$.
    \item If the claim holds for $(N-1,k-1)$ then it holds for $(N,k)$.
\end{itemize}
Combining both bullets proves the claim for any pair $(N,k)$, through the chain of implications
\[
(N-k+1,1)\implies(N-k+2,2)\implies\dots\implies(N,k)~.\]
For the first bullet, fix any $N\geq2$. We need to prove that $h^{N}_{\sigma_1,\dots,\sigma_{N}}(x)$ does not depend on $\sigma_{2},\dots,\sigma_{N}$ for $x\notin
\left[\frac{\sigma_{1}+1}{4},
\frac{\sigma_{1}+1}{4}+\frac{1}{4}\right]$. Indeed, if $\sigma_1=0$ then by construction $h^{N}_{0,\dots,\sigma_{N}}(x)$ does not depend on $\sigma_2,\dots,\sigma_{N}$ for $x\notin\left[\frac{1}{4},\frac{1}{2}\right]$. Similarly, if $\sigma_1=1$ then by construction $h^{N}_{1,\dots,\sigma_{N}}(x)$ does not depend on $\sigma_2,\dots,\sigma_{N}$ for $x\notin\left[\frac{1}{2},\frac{3}{4}\right]$.

For the second bullet, assume the claim is true for some pair $(N-1,k-1)$. By renaming the variables $\sigma_{i}\leftarrow\sigma_{i+1}$, the induction hypothesis states that $h^{N-1}_{\sigma_2,\dots,\sigma_{N}}(x)$ does not depend on $\sigma_{k+1},\dots,\sigma_{N}$ for
$x\notin[\sum_{i=1}^{k-1}\frac{\sigma_{i+1}+1}{4^{i}},$ $\sum_{i=1}^{k-1}\frac{\sigma_{i+1}+1}{4^{i}}+\frac{1}{4^{k-1}}]$. Thus, $h^{N-1}_{\sigma_2,\dots,\sigma_{N}}(4x-1-\sigma_{1})$ does not depend on $\sigma_{k+1},\dots,\sigma_{N}$ for
\begin{align*}
    (4x-1-\sigma_{1})\notin&\left[\sum_{i=1}^{k-1}\frac{\sigma_{i+1}+1}{4^{i}}, \sum_{i=1}^{k-1}\frac{\sigma_{i+1}+1}{4^{i}}+\frac{1}{4^{k-1}}\right]
    \\
    \iff x\notin&\left[\sum_{i=1}^{k}\frac{\sigma_{i}+1}{4^{i}}, 
    \sum_{i=1}^{k}\frac{\sigma_{i}+1}{4^{i}}+\frac{1}{4^{k}}\right]~.
\end{align*}
Noticing that by definition, $h^{N}_{\sigma_1,\sigma_2,\dots,\sigma_{N}}(x)$ depends on $\sigma_2,\dots,\sigma_N$ only when $(4x-1-\sigma_{i})$ is fed to $h^{N-1}_{\sigma_2,\dots,\sigma_{N}}$ finishes the proof.
\end{proof}

From now on we fix some $N$ we will specify later, and abbreviate $h_\sigma=h^N_{\sigma}$. Let $\sigma\sim\mathrm{Unif}(\{0,1\}^{N})$, and consider the random sequence $x_1,x_2,\dots$ produced by $\Acal$ when applied to $h_{\sigma}$ (where the randomness comes from the random choice of $\sigma$).
\begin{lemma}
For any $t\in\NN$, $1\leq l<k\leq N:\Pr_{\sigma}\left[
x_{t+1}\in I_{\sigma_{1},\dots,\sigma_{l},\dots,\sigma_{k}}
\middle|x_1,\dots,x_{t}\notin I_{\sigma_{1},\dots,\sigma_{l}}
\right]
\leq \frac{1}{2^{k-l-1}}$.
\end{lemma}
\begin{proof}
If $x_1,\dots,x_{t}\notin I_{\sigma_{1},\dots,\sigma_{l}}$, then \lemref{lem: h outside I} tells us that $\OO_{h_\sigma}(x_1),\dots,\OO_{h_\sigma}(x_t)$ do not depend on $\sigma_{l+1},\dots,\sigma_{N}$. 
Since $x_{t+1}$ is some deterministic function of $\OO_{h_\sigma}(x_1),\dots,\OO_{h_\sigma}(x_t)$, which is the information that $\Acal$ obtains along it's first $t$ iterations, we can deduce that in this case 
\begin{enumerate}
    \item Conditioning on past information, $\sigma_{l+1},\dots,\sigma_{N}\sim\mathrm{Unif}(\{0,1\})$. In particular we get $(\sigma_{l+1},\dots,\sigma_{k-1})\sim\mathrm{Unif}(\{0,1\}^{k-l-1})$.
    \item $x_{t+1}$ is chosen independently of $\sigma_{l+1},\dots,\sigma_{k-1}$.
\end{enumerate}
Note that by \lemref{lem: I disjoint}, $I_{\sigma_{1},\dots,\sigma_{l},\sigma_{l+1},\dots,\sigma_{k-1},\sigma_{k}},I_{\sigma_{1},\dots,\sigma_{l},\sigma'_{l+1},\dots,\sigma'_{k-1},\sigma'_{k}}$ are disjoint for any $(\sigma_{l+1},\dots,\sigma_{k-1})\neq(\sigma'_{l+1},\dots,\sigma'_{k-1})$, thus the events $x\in I_{\sigma_{1},\dots,\sigma_{l},\sigma_{l+1},\dots,\sigma_{k-1},\sigma_{k}}$ and $x\in I_{\sigma_{1},\dots,\sigma_{l},\sigma'_{l+1},\dots,\sigma'_{k-1},\sigma'_{k}}$ are disjoint. Since there are $2^{(k-1)-(l+1)+1}=2^{k-l-1}$ such binary vectors, we obtain $2^{k-l-1}$ disjoint events with equal probabilities. Thus, their probability is at most $\frac{1}{2^{k-l-1}}$, which proves the claim.
\end{proof}

Recall that $I_{\sigma_{1}}\supset
I_{\sigma_1,\sigma_{2}}\supset\dots\supset I_{\sigma_1,\dots,\sigma_{N}}$ by \lemref{lem: I inclusion}. Combined with the previous lemma, we get
\begin{align*}
    &\Pr_{\sigma}\left[\exists t\in[T_{0}]:x_{t}=x^*\right]
    \leq
    \Pr_{\sigma}\left[\exists t\in[T_0]:x_{t}\in I_{\sigma_1,\dots,\sigma_{N}}\right]
    \\
    &\leq\sum_{t=1}^{T_0}\Pr_{\sigma}\left[x_{t}\in I_{\sigma_1,\dots,\sigma_{N}}\right]
    \\
    &\leq
    \sum_{t=1}^{T_0}\Pr_{\sigma}\left[\exists s\in[t],l\in[N]:
    x_1,\dots,x_{s}\notin I_{\sigma_1,\dots,\sigma_{l}}
    \land
    x_{s+1}\in I_{\sigma_1,\dots,\sigma_{l+\frac{N}{t}}}
    \right]
    \\
    &\leq\sum_{t=1}^{T_0}\sum_{s=1}^{t}\sum_{l=1}^{N}\Pr_{\sigma}\left[x_1,\dots,x_{s}\notin I_{\sigma_1,\dots,\sigma_{l}}
    \land
    x_{s+1}\in I_{\sigma_1,\dots,\sigma_{l+\frac{N}{t}}}
    \right]
    \\
    &=\sum_{t=1}^{T_0}\sum_{s=1}^{t}\sum_{l=1}^{N}\Pr_{\sigma}\left[x_{s+1}\in I_{\sigma_1,\dots,\sigma_{l+\frac{N}{t}}}
    \middle| x_1,\dots,x_{s}\notin I_{\sigma_1,\dots,\sigma_{l}}
    \right]
    \cdot\Pr_{\sigma}\left[x_1,\dots,x_{s}\notin I_{\sigma_1,\dots,\sigma_{l}}\right]
    \\
    &\leq\sum_{t=1}^{T_0}\sum_{s=1}^{t}\sum_{l=1}^{N}\frac{1}{2^{\frac{N}{t}-1}}\cdot1
    =\sum_{t=1}^{T_0}\frac{Nt}{2^{\frac{N}{t}-1}}
    \leq T_{0}\cdot\frac{N{T_0}}{2^{\frac{N}{T_0}-1}}~~.
\end{align*}
Since $\frac{N(T_0)^2}{2^{\frac{N}{T_0}-1}}\overset{N\to\infty}{\longrightarrow}0$, there exists some finite $N$ (which depends on $T_0$) such that $\frac{N{T_0^2}}{2^{\frac{N}{T_0}-1}}
\leq \frac{1}{4}$. With this $N$, we get
\[
\Pr_{\sigma}\left[\exists t\in[T_0]:x_{t}=x^*\right]
\leq \frac{1}{4}<\frac{1}{2}~,
\]
proving \eqref{eq: Yao hard distribution} and finishing the proof.

\section{Technical lemmas}\label{app:technical lemmas}

\begin{lemma} \label{lemma: L>1/eps}
    Denote by $f(\cdot)$ the $L_0$-Lipschitz function $\bx\mapsto L_0 |x_1|$. Assume $\tilde{f}(\cdot)$ has $L$-Lipschitz gradients, and satisfies $\left\|f-\tilde{f}\right\|_{\infty}\leq\epsilon$. Then $L\geq\frac{L_0}{8\epsilon}$.
\end{lemma}
\begin{proof}
    Due to rescaling we can assume without loss of generality that $L_0=1$. Denoting by $\mathbf{e}_1$ the first standard basis vector, we have
    \begin{align*}
    \tilde{f}\left(-4\epsilon\cdot\mathbf{e}_1\right)&\geq f\left(-4\epsilon\cdot\mathbf{e}_1\right)-\epsilon
    =4\epsilon-\epsilon=3\epsilon~,\\
    \tilde{f}\left(4\epsilon\cdot\mathbf{e}_1\right)&\geq f\left(4\epsilon\cdot\mathbf{e}_1\right)-\epsilon
    =4\epsilon-\epsilon=3\epsilon~,\\
    \tilde{f}\left(\mathbf{0}\right)&\leq f\left(\mathbf{0}\right)+\epsilon=\epsilon~.
    \end{align*}
    By the mean value theorem, there exist $-4\epsilon<t_0<0,~0<t_1<4\epsilon$ such that
    \begin{align*}
        \frac{\partial}{\partial x_1}\tilde{f}\left(t_0\right)
        &=\frac{\tilde{f}\left(\mathbf{0}\right)- \tilde{f}\left(-4\epsilon\cdot\mathbf{e}_1\right)}{4\epsilon}
        \leq\frac{\epsilon-3\epsilon}{4\epsilon}=-\frac{1}{2}~,\\
        \frac{\partial}{\partial x_1}\tilde{f}\left(t_1\right)
        &=\frac{\tilde{f}\left(4\epsilon\cdot\mathbf{e}_1\right)-\tilde{f}\left(\mathbf{0}\right)}{4\epsilon}
        \geq\frac{3\epsilon-\epsilon}{4\epsilon}=\frac{1}{2}~.
    \end{align*}
    So by Cauchy-Schwarz and $L$-smoothness of $\tilde{f}$:
    \begin{align*}
    1&=
    \left|\frac{\partial}{\partial x_1}\tilde{f}\left(t_1\right)-\frac{\partial}{\partial x_1}\tilde{f}\left(t_0\right)\right|
    =\left|\left\langle \nabla\tilde{f}\left(t_1\cdot\mathbf{e}_1\right)-\nabla\tilde{f}\left(t_0\cdot\mathbf{e}_1\right),\mathbf{e}_1 \right\rangle\right|\\
    &\leq
    \left\|\nabla\tilde{f}\left(t_1\cdot\mathbf{e}_1\right)-\nabla\tilde{f}\left(t_0\cdot\mathbf{e}_1\right)\right\|\leq L\left|t_1-t_0\right|\leq L\cdot8\epsilon
    ~.
    \end{align*}

\end{proof}
\begin{lemma} \label{lemma:Lsmooth approx is L-Lip}
    If $\tilde{f}(\cdot)$ has $L$-Lipschitz gradients and satisfies $\|f-\tilde{f}\|_{\infty}\leq\epsilon$ for some 1-Lipschitz function $f(\cdot)$, then for all $\bx\in\reals^d:\ \|\nabla\tilde{f}(\bx)\|\leq 1+2\epsilon+\frac{L}{2}$.
\end{lemma}
\begin{proof}
    Let $\bx,\by\in\reals^d$. Denote $\gamma(t):=(1-t)\cdot\bx+t\cdot\by$, and notice that
    \begin{align} \label{eq:Taylor appendix lemma}
        \tilde{f}\left(\by\right)-\tilde{f}\left(\bx\right)
        &=\tilde{f}\left(\gamma\left(1\right)\right)-\tilde{f}\left(\gamma\left(0\right)\right)
        =\int_{0}^{1}\left(\tilde{f}\circ\gamma\right)'\left(t\right)dt
        =\int_{0}^{1}\left\langle \nabla\tilde{f}\left(\gamma\left(t\right)\right),\gamma'\left(t\right)\right\rangle dt \nonumber\\
        &=\int_{0}^{1}\left\langle \nabla\tilde{f}\left(\gamma\left(t\right)\right),{\by-\bx}\right\rangle dt~.
    \end{align}
    Combining Cauchy-Schwarz with the fact that $\nabla\tilde{f}$ is $L$-Lipschitz, we get
    \begin{gather*} 
        \left\langle \nabla\tilde{f}\left(\bx\right)-\nabla\tilde{f}\left(\gamma\left(t\right)\right),\by-\bx\right\rangle 
        \leq\left\| \nabla\tilde{f}\left(\bx\right)-\nabla\tilde{f}\left(\gamma\left(t\right)\right)\right\| \cdot\left\|\by-\bx\right\| 
        \leq L\left\|\bx-\gamma\left(t\right)\right\|\cdot\left\|\by-\bx\right\| \\
        \implies\left\langle \nabla\tilde{f}\left(\gamma\left(t\right)\right),\by-\bx\right\rangle
        \geq\left\langle \nabla\tilde{f}\left(\bx\right),\by-\bx\right\rangle -L\left\| \gamma\left(t\right)-\bx\right\| \cdot\left\|\by-\bx\right\|~.
    \end{gather*}
    Plugging this into \eqref{eq:Taylor appendix lemma} gives
    \begin{align*}
        \tilde{f}\left(\by\right)-\tilde{f}\left(\bx\right)
        \geq&\int_{0}^{1}\left(\left\langle \nabla\tilde{f}\left(\bx\right),\by-\bx\right\rangle -L\left\| \gamma\left(t\right)-\bx\right\| \cdot\left\|\by-\bx\right\|\right)dt\\
        =&\left\langle \nabla\tilde{f}\left(\bx\right),\by-\bx\right\rangle -L\left\|\by-\bx\right\| \cdot\int_{0}^{1}\left\| \gamma\left(t\right)-\bx\right\| dt\\
        =&\left\langle \nabla\tilde{f}\left(\bx\right),\by-\bx\right\rangle -L\left\| \by-\bx\right\| \cdot\left[\frac{1}{2}\left\| \gamma\left(1\right)-\bx\right\| ^{2}-\frac{1}{2}\left\| \gamma\left(0\right)-\bx\right\|^{2}\right]\\
        =&\left\langle \nabla\tilde{f}\left(\bx\right),\by-\bx\right\rangle -\frac{L}{2}\left\| \by-\bx\right\|^{3}\\
        \implies&\left\langle \nabla\tilde{f}\left(\bx\right),\by-\bx\right\rangle
        \leq\tilde{f}\left(\by\right)-\tilde{f}\left(\bx\right)+\frac{L}{2}\left\|\by-\bx\right\|^{3}~.
    \end{align*}
    We assume $\|\nabla\tilde{f}(\bx)\|\neq0$ since otherwise the desired claim is trivial. In particular, if $\by=\bx+\frac{\nabla\tilde{f}\left(\bx\right)}{\left\|\nabla\tilde{f}\left(\bx\right)\right\|}$ then $\|\by-\bx\|=1$ and inequality above reveals
    \[
    \left\|\nabla\tilde{f}\left(\bx\right)\right\| \leq\tilde{f}\left(\by\right)-\tilde{f}\left(\bx\right)+\frac{L}{2}\leq f\left(\by\right)-f\left(\bx\right)+2\epsilon+\frac{L}{2}
    \leq \left\| \by-\bx\right\|+2\epsilon+\frac{L}{2}
    =1+2\epsilon+\frac{L}{2}~,
    \]
    where we used the fact that $\|\tilde{f}-f\|_\infty\leq\epsilon$, and that $f$ is 1-Lipschitz.
\end{proof}

\end{document}